\documentclass[a4paper,11pt]{amsart}
\usepackage{amssymb,hyperref, xcolor}
\usepackage[margin=2.5cm]{geometry}
\newtheorem{theorem}{Theorem}
\newtheorem{lemma}{Lemma}[section]
\newtheorem{proposition}[lemma]{Proposition}
\newcommand{\rc}{\mathrm{RC}}
\renewcommand{\H}{\mathrm{H}}
\numberwithin{equation}{section}
\usepackage{etoolbox}
\makeatletter
\patchcmd{\@maketitle}
  {\ifx\@empty\@dedicatory}
  {\ifx\@empty\@date \else {\vskip3ex \centering\footnotesize\@date\par\vskip1ex}\fi
   \ifx\@empty\@dedicatory}
  {}{}
\patchcmd{\@adminfootnotes}
  {\ifx\@empty\@date\else \@footnotetext{\@setdate}\fi}
  {}{}{}
\makeatother

\begin{document}
\title[Relational complexity of affine groups]{On the maximum relational complexity\\ of affine groups}

\author[C.~del Valle]{Coen del Valle}
\address{School of Mathematics and Statistics, The Open University, Milton Keynes, UK}
\email{Coen.del-Valle@open.ac.uk}

\author[N.~Gill]{Nick Gill}
\address{School of Mathematics and Statistics, The Open University, Milton Keynes, UK}
\email{nick.gill@open.ac.uk}

\author[P.~Spiga]{Pablo Spiga}
\address{Dipartimento di Matematica e Applicazioni, University of Milano-Bicocca, Via Cozzi 55, 20125 Milano, Italy}
\email{pablo.spiga@unimib.it}
\thanks{This research is supported by the Engineering and Physical Sciences Research Council (EPSRC) [grant number EP/Z534742/1]. CdV and NG would like to thank the Dipartimento di Matematica e Applicazioni at the University of Milano-Bicocca, for hosting them during the initiation of this work.}
\date{\today}
\begin{abstract}
Let $G$ be a group of permutations of a set $\Omega$. The relational complexity $\mathrm{RC}(G)$ is the least integer $r\geq 2$ such that the $G$-orbit of every finite tuple is determined by the $G$-orbits of its $r$-subtuples; there is currently only a small collection of groups $G$ for which $\mathrm{RC}(G)$ is known precisely. In this paper we study the relational complexity of finite affine groups acting naturally on a $d$-dimensional vector space over $\mathbb{F}_q$. We prove that every such group has relational complexity at most $d+2$ when $q>2$ and at most $d+1$ when $q=2$; we produce a complete classification of the primitive such groups for which these bounds are attained.
\end{abstract}



\maketitle
\vspace{-0.3cm}\hrule
\section{Introduction}

Let $G$ be a group of permutations of a finite set $\Omega$. Informally, the \textit{\textbf{relational complexity}} of $G$ is the least integer $r\geq 2$ such that the orbits of the induced action of $G$ on $\Omega^r$ determine the orbits of $G$ on $\Omega^m$ for all $m\geq r$. Over recent years, there has been significant interest in studying the relational complexity of various important families of permutation groups. Much of this intrigue has been sparked by conjectures and problems posed by G. Cherlin, several of which can be found in~\cite{Cherlin2000}. Among these problems is \emph{Cherlin's Conjecture}, which proposes a classification of the finite primitive permutation groups with relational complexity two -- that is, relational complexity as small as possible -- the conjecture was proved by the second and third authors along with M.W. Liebeck in~\cite{GLieSp}. 

The current paper is concerned with the other end of the spectrum; specifically, we are interested in the question \emph{``for which primitive permutation groups is the relational complexity as big as possible?"} In order to make this question precise, we must first introduce a few definitions.

For $\Lambda \subseteq \Omega$, we say that $\Lambda$ is an \textit{\textbf{independent set}} if its pointwise stabilizer is strictly smaller than the pointwise stabilizer of every proper subset of $\Lambda$. We define the \textit{\textbf{height}} of $G$ to be the maximum size of an independent set, and we denote this quantity by $\H(G)$.

We now formally introduce the relational complexity of the permutation group $G$. Let $r,n \in \mathbb{N}$ with $r \le n$. If $I,J \in \Omega^n$, then we write
$$
I \sim_r J,
$$
and say that $I$ and $J$ are $r$-equivalent with respect to the action of $G$ if, for every subtuple of size $r$ of $I$, there exists an element of $G$ mapping it to the corresponding subtuple of $J$. More precisely,
for every choice of distinct $k_1,k_2,\ldots,k_r \in \{1,\ldots,n\}$, there exists $h \in G$ such that $I_{k_i}^h=J_{k_i}$ for every $i \in \{1,\ldots,r\}$. Here, $I_k$ denotes the $k^{\rm th}$ entry of the tuple $I$, while $I^h$ denotes the image of $I$ under the action of $h$.

The group $G$ is said to have \textit{\textbf{relational complexity}} $r$ if $r$ is the smallest integer at least 2 such that, for every $n \in \mathbb{N}$ with $n \ge r$ and all $n$-tuples $I,J \in \Omega^n$,
$$
I \sim_r J \hbox{ implies } I \sim_n J.
$$
In this case, we write $\rc(G)=r$.

In 2016, Cherlin~\cite{Cherlin2016} asked which primitive permutation groups of degree $n$ have relational complexity much larger than $\log n$. It was shown in~\cite{GLS} and~\cite{KelRo} that such groups must be large-base -- that is, groups $G$ satisfying $\mathrm{Alt}(m)^\ell\leq G\leq \mathrm{Sym}(m)\wr \mathrm{Sym}(\ell)$ where the action of $\mathrm{Sym}(m)$ is on the $k$-element subsets of $\{1,2,\dots,m\}$ and the wreath product is equipped with its product action -- and very recently P. Wu showed that every large-base group of degree ${m\choose k}^\ell$ has relational complexity at most $m\ell$~\cite{WuThesis}.

In light of these results, it is natural to ask how large the relational complexity of a primitive permutation group can be (in terms of the degree) for the different families given by the O'Nan--Scott theorem. We answer this question for the family of primitive permutation groups of affine type.

\subsection{Main results} To state our main results we shall introduce some notation. Let $d$ and $f$ be positive integers, let $p$ be a prime, and set $q=p^f$. 
Let $V=\mathbb{F}_q^d$ be the $d$-dimensional vector space of row vectors over the finite field $\mathbb{F}_q$ of order $q$. Then
$$
\mathrm{AGL}_d(q)=V\rtimes \mathrm{GL}_d(q)
$$
acts naturally on $V$. More precisely, for $\omega\in V$ and $(v,A)\in \mathrm{AGL}_d(q)$, where $v\in V$ and $A\in\mathrm{GL}_d(q)$, we have
$$
\omega^{(v,A)}=v+\omega A.
$$

Our first main result is the following, some of which has already been established in the literature but is restated for completeness (see e.g.~\cite{Cherlin2000}).
\begin{theorem}\label{thrm:main}
Let $d$ be a positive integer and $q$ a prime power. Then
$$
\H(\mathrm{AGL}_d(q))=
\begin{cases}
d+1 & \textrm{when }(d,q)\ne(1,2),\\
1   & \textrm{when }(d,q)=(1,2).
\end{cases}
$$
Moreover,
$$
\rc(\mathrm{AGL}_d(q))=
\begin{cases}
d+2 & \textrm{when }q>2\textrm{ and }(d,q)\ne(1,3),\\
d+1 & \textrm{when }q=2\textrm{ and }d\textrm{ is odd},
      \textrm{ or when }(d,q)=(1,3),\\
d   & \textrm{when }q=2\textrm{ and }d\textrm{ is even},\\
\end{cases}
$$
Moreover, every subgroup of $\mathrm{AGL}_d(q)$ has relational complexity at most $d+2$ when $q>2$, and at most $d+1$ when $q=2$.
\end{theorem}

Given that our main interest is in relational complexity, we should explain why Theorem~\ref{thrm:main} includes a bound for the height. This is because of a fundamental inequality relating $\H(G)$ and $\rc(G)$. By~\cite[Lemma~2.1]{GLS}, if $G\ne 1$ then we have
\begin{equation}\label{eq:2.1}
\rc(G) \le \H(G)+1.
\end{equation}
Although the proof of this inequality is elementary, the result is particularly useful. Relational complexity does not behave well under passage to subgroups. In general, if $G$ is a permutation group and $K$ is a proper subgroup of $G$, there need not be any relation between $\rc(G)$ and $\rc(K)$. By contrast, height is monotone under taking subgroups: if $K\le G$, then $\H(K)\le \H(G)$; see~\cite[Lemma~2.2]{GLS}. Consequently, whenever $1\ne K\le G$,~\eqref{eq:2.1} yields
\begin{equation}\label{eq:height sub}
\rc(K)\le \H(K)+1\le \H(G)+1.
\end{equation}
We will make use of this observation throughout the paper.



Our remaining results classify the primitive groups of affine type that achieve the upper bounds on relational complexity given in the final sentence of Theorem~\ref{thrm:main}, namely $d+2$ with $q>2$ and $d+1$ with $q=2$.

\begin{theorem}\label{thrm:main2}
Let
$G=V\rtimes G_0\le \mathrm{AGL}_d(q)$,
where $G_0$ is the stabilizer in $G$ of the zero vector and acts irreducibly on $V$. Suppose that
$\rc(G)=d+2$.
Then $q>2$ and one of the following holds:
\begin{enumerate}
\item\label{part1} $d=1$ and $|G_0|>2$,
\item\label{part2} $d\ge 2$ and there exists a divisor $f_1$ of $f$ such that
\begin{align*}
&\mathrm{SL}_d(p^{f_1})\unlhd G_0
\le \mathrm{GL}_d(p^{f_1})\mathbb{F}_q^\ast,\\
&\mathrm{SL}_d(p^{f_1})<G_0\cap \mathrm{GL}_d(p^{f_1}),
\end{align*}
where $\mathbb{F}_q^\ast$ is identified with the group of non-zero scalar matrices in $\mathrm{GL}_d(q)$.

\item\label{part3} $d= 2$, $p=3$, $f$ is even and
$G_0=\mathrm{SL}_2(5)Z$, where $\mathrm{SL}_2(5)$ is a maximal subgroup of $\mathrm{SL}_2(9)$ and $Z\le\mathbb{F}_q^\ast$ is a group of scalar matrices. Moreover, $G_0$ contains a scalar matrix of order $4$.
\end{enumerate}

The converse also holds, that is, if $G_0$ is as in parts~\eqref{part1},~\eqref{part2} or~\eqref{part3}, then $G=V\rtimes G_0$ has relational complexity $d+2$.
\end{theorem}
\begin{theorem}\label{thrm:main3}
Let
$G=V\rtimes G_0\le \mathrm{AGL}_d(2)$,
where $G_0$ is the stabilizer in $G$ of the zero vector and acts irreducibly on $V$. Suppose that
$\rc(G)=d+1$.
Then one of the following holds:
\begin{enumerate}
\item\label{part12} $d$ is odd and $G_0=\mathrm{SL}_d(2)$,
\item\label{part22} $d\geq 2$ is even and $G_0$ is the natural image of either $\mathrm{Alt}(d+1)$ or $\mathrm{Alt}(d+2)$ acting on the fully deleted permutation module,
\item\label{part32} $d\ge 6$, $d$ is even and $G_0$ is either $\Omega_d^+(2)$ or $\Omega_d^-(2)$,
\item\label{part42} $d=4$ and $G_0=\mathrm{Alt}(7)$.
\end{enumerate}

The converse also holds, that is, if $G_0$ is as in parts~\eqref{part12},~\eqref{part22},~\eqref{part32} or~\eqref{part42}, then $G=V\rtimes G_0$ has relational complexity $d+1$.
\end{theorem}
Theorems~\ref{thrm:main2} and~\ref{thrm:main3} apply, in particular, to all primitive groups of affine type: a group $G$ is primitive of affine type if and only if there exists a prime $p$ and an integer $d\geq 1$ such that $G$ is permutation isomorphic to a subgroup $H\leq\mathrm{AGL}_d(p)$ containing the translation subgroup $V=\mathbb{F}_p^d$ and the stabiliser $H_0$ of the zero vector acts irreducibly on $V$.


\subsection{Structure of the paper} In Section~\ref{sec:ais} we begin with some general background in affine geometry and the theory of relational complexity, before proving Theorem~\ref{thrm:main}. In Section~\ref{sec:maximal}, we prove the forward implication of Theorem~\ref{thrm:main2}; the converse is proved in Section~\ref{sec:converse}. The final two sections are organized in the reverse order: in Section~\ref{sec:someq=2} we explicitly calculate the relational complexities of some groups which are relevant to the proof of Theorem~\ref{thrm:main3}, and in Section~\ref{sec:maximalq=2} we complete the proof of Theorem~\ref{thrm:main3}.

\section{Independent sets and almost independent sets}\label{sec:ais}

Throughout this section let $e_1,\ldots,e_d$ denote the standard basis of $V=\mathbb{F}_q^d$. An \textit{\textbf{affine subspace}} is a set $\pi$ of \textit{\textbf{points}} of the form $P+w$, where $P\in V$ is fixed and $w$ ranges over a vector subspace $W$ of $V$. The dimension of $\pi$ is defined to be the vector space dimension $\dim_{\mathbb{F}_q}W$ of $W$. By abuse of notation, we write $\dim \pi$ for this dimension.

\subsection{Frames and independent sets} If $P_1,P_2,\ldots,P_r$ are points of $\pi$, then a straightforward calculation shows that
$$
a_1P_1+a_2P_2+\cdots+a_rP_r\in\pi,
$$
for all $a_1,a_2,\ldots,a_r\in\mathbb{F}_q$ satisfying
$
a_1+a_2+\cdots+a_r=1.$
A \textit{\textbf{frame}} of the affine subspace $\pi$ is a set of points $\{P_1,P_2,\ldots,P_r\}$ such that every point of $\pi$ can be written uniquely in the form
$
a_1P_1+a_2P_2+\cdots+a_rP_r,
$
where
$
a_1+a_2+\cdots+a_r=1.
$
Observe that every frame of $\pi$ contains exactly $1+\dim \pi$ points. We will also use the fact that $\mathrm{AGL}_d(q)$ acts regularly, that is, transitively and freely, on the set of ordered frames of the affine space $V$.

\begin{lemma}\label{l:0}
Let $\emptyset\ne \Lambda\subseteq V$. Then $\Lambda$ is an independent set for the action of $\mathrm{AGL}_d(q)$ if and only if $\Lambda$ is a frame for some affine subspace of $V$, with the single exception that, when $(d,q)=(1,2)$, the frame $\Lambda=V$ is not independent.
\end{lemma}

\begin{proof}
Suppose first that $\Lambda$ is a frame for an affine subspace $\pi$ of $V$ and that $\Lambda\ne V$ when $(d,q)=(1,2)$. Up to replacing $\Lambda$ with a suitable $\mathrm{AGL}_d(q)$-conjugate, we have
$$
\Lambda=\{0,e_1,\ldots,e_\ell\}
$$
for some $0\le\ell\le d$. A direct calculation shows that the pointwise stabilizer of $\Lambda$ is strictly contained in the pointwise stabilizer of each subset obtained by deleting one point from $\Lambda$. Thus $\Lambda$ is an independent set for the action of $\mathrm{AGL}_d(q)$ on $V$.

Conversely, suppose that $\Lambda$ is an independent set for the action of $\mathrm{AGL}_d(q)$ on $V$. Up to replacing $\Lambda$ with a suitable $\mathrm{AGL}_d(q)$-conjugate, we may suppose that $0\in\Lambda$. Thus $\Lambda=\{0,v_1,\ldots,v_\ell\}$,
for some $v_1,\ldots,v_\ell\in V\setminus\{0\}$. It follows from the independence of $\Lambda$ that $\{v_1,\ldots,v_\ell\}$ is an independent set for the action of $\mathrm{GL}_d(q)$ on $V$. Hence the vectors $v_1,\ldots,v_\ell$ are linearly independent. Therefore, $\Lambda$ is a frame for the affine subspace
$$
0+\langle v_1,\ldots,v_\ell\rangle.
$$
Finally, when $(d,q)=(1,2)$, the pointwise stabilizers of $V$ and of either singleton subset of $V$ are all trivial, so $V$ is not independent.
\end{proof}

\subsection{Almost independent sets in \texorpdfstring{V}{$V$}}
Let $G$ be a permutation group on a set $\Omega$. We say that a subset, $\Lambda,$ of $\Omega$ is \textit{\textbf{almost independent}} for $G$ if, for every $\lambda\in\Lambda$, the set $\Lambda\setminus\{\lambda\}$ is independent for $G$. (This notion was introduced by Scott Hudson in his PhD thesis \cite{ScottHudson}.)

\begin{lemma}[{\cite[Lemma~2.4.2]{ScottHudson}}]\label{l: almost independent}
Suppose that $m$ is a positive integer, that $I,J\in \Omega^{m+1}$, that $I\sim_{m} J$ but $I\not\sim_{m+1} J$. Then the underlying set of $I$ (and, therefore, of $J$) is an almost independent set.
\end{lemma}

We make three general remarks. First, note that independent sets are almost independent. Second, it is clear that the maximum size of an almost independent set is either $\H(G)$ or $\H(G)+1$. Lemma~\ref{l: almost independent} implies that if $\rc(G)=\H(G)+1$, then there exists an almost independent set of cardinality $\H(G)+1$.

Third, observe that, like independent sets, almost independent sets behave well with respect to subgroups: if $H\leq G$ and $\Gamma$ is almost independent for $H$, then $\Gamma$ is almost independent for $G$. Consequently the maximum size of an almost independent set for $H$ is less than or equal to the maximum size of an almost independent set for $G$.

Now we return to the situation where $G\leq \mathrm{AGL}_d(q)$. Our aim is to provide a description of the ``large'' almost independent subsets of $V$. By the third remark we can (and will) take $G=\mathrm{AGL}_d(q)$ for the rest of this subsection.

We first record an observation that will be used repeatedly. Let
$$
x=\sum_{i=1}^r\alpha_i e_i,
\qquad 1\le r\le d,
$$
with $x\not\in \{0,e_1,\ldots,e_r\}$.
Then $\{0,e_1,\ldots,e_r,x\}$
is almost independent if and only if
$\alpha_i\ne 0$ for every $i\in\{1,\ldots,r\}$
and $\sum_{i=1}^r\alpha_i\ne 1$.
Indeed, deleting $x$ leaves a frame. Deleting $e_i$ leaves a frame if and only if $\alpha_i\ne 0$, while deleting $0$ leaves a frame if and only if
$\sum_{i=1}^r\alpha_i\ne 1$.

\begin{lemma}\label{l:aiodd}
Suppose that $q>2$, and let $\Lambda$ be an almost independent subset of $V$ of cardinality at least $d+2$. Then $|\Lambda|=d+2$, and $\Lambda$ is $\mathrm{AGL}_d(q)$-conjugate to a set of the form
$$
\left\{0,e_1,\ldots,e_d,\sum_{i=1}^d\alpha_i e_i\right\},
$$
where $\alpha_i\ne 0$ for every $i$ and $\sum_{i=1}^d\alpha_i\ne 1$.
\end{lemma}

\begin{proof}
For every $\lambda\in\Lambda$, the set $\Lambda\setminus\{\lambda\}$ is independent. By Lemma~\ref{l:0}, it has cardinality at most $d+1$. Therefore,
$|\Lambda|\le d+2$,
and hence $|\Lambda|=d+2$.

Fix $\lambda\in\Lambda$. The set $\Lambda\setminus\{\lambda\}$ is a frame for $V$. After replacing $\Lambda$ by a suitable $\mathrm{AGL}_d(q)$-conjugate, we may assume that
$$
\Lambda\setminus\{\lambda\}=\{0,e_1,\ldots,e_d\}.
$$
Write 
$\lambda=\sum_{i=1}^d\alpha_i e_i$.
The result now follows from the observation preceding the lemma.
\end{proof}

\begin{lemma}\label{l:ai}
Suppose that $q=2$ and $d\ge 2$, and let $\Lambda$ be an almost independent subset of $V$ of cardinality at least $d+1$. Then $\Lambda$ is $\mathrm{AGL}_d(2)$-conjugate to one of the following:
\begin{enumerate}
\item\label{l:ai1}$
\{0,e_1,\ldots,e_d\}$,
in which case $\Lambda$ is independent;

\item\label{l:ai2}
$\left\{0,e_1,\ldots,e_{d-1},
\sum_{i=1}^{d-1}e_i\right\}$,
where $d$ is odd;

\item\label{l:ai3}
$
\left\{0,e_1,\ldots,e_d,
\sum_{i=1}^d e_i\right\}$,
where $d$ is even.
\end{enumerate}
\end{lemma}

\begin{proof}
As in the proof of Lemma~\ref{l:aiodd}, we have
$|\Lambda|\le d+2$.

Suppose first that $|\Lambda|=d+1$. If $\Lambda$ is independent, then by Lemma~\ref{l:0} it is a frame for $V$ and is therefore conjugate to the set in~\eqref{l:ai1}.

Suppose that $\Lambda$ is not independent. Choose $\lambda\in\Lambda$. Since $\Lambda\setminus\{\lambda\}$ is an independent set of cardinality $d$, it is a frame for an affine subspace of dimension $d-1$. After replacing $\Lambda$ by a suitable conjugate, we may assume that
$$
\Lambda\setminus\{\lambda\}
=
\{0,e_1,\ldots,e_{d-1}\}.
$$
Since $\Lambda$ is not independent, $\lambda$ belongs to the affine span of these points. Thus
$$
\lambda=\sum_{i=1}^{d-1}\alpha_i e_i.
$$
The observation preceding Lemma~\ref{l:aiodd} shows that every $\alpha_i$ is non-zero and that
$$
\sum_{i=1}^{d-1}\alpha_i\ne 1.
$$
Since $q=2$, we have $\alpha_i=1$ for every $i$. The second condition is therefore equivalent to $d-1$ being even, or equivalently to $d$ being odd. This gives the set in~\eqref{l:ai2}.

Suppose now that $|\Lambda|=d+2$. As in the proof of Lemma~\ref{l:aiodd}, we may assume that
$$
\Lambda=
\left\{0,e_1,\ldots,e_d,
\sum_{i=1}^d\alpha_i e_i\right\}.
$$
Once again, every $\alpha_i$ is non-zero, and hence $\alpha_i=1$ for every $i$. Moreover,
$$
\sum_{i=1}^d\alpha_i\ne 1
$$
if and only if $d$ is even. This gives the set in~\eqref{l:ai3}.
\end{proof}

\subsection{Proving Theorem~\ref{thrm:main}}
Because of the peculiar behaviour of the relational complexity of $\mathrm{AGL}_d(2)$, before proving Theorem~\ref{thrm:main} we need a technical lemma.
\begin{lemma}\label{l:binarysubgroups}
Let $K\le \mathrm{AGL}_d(2)$. Then
$\rc(K)\le d+1$.
\end{lemma}

\begin{proof}
The case $d=1$ is immediate, so assume $d\ge 2$.
Suppose that there exist $I$ and $J$ with $I\sim_{d+1}J$ and $I\not\sim_{d+2}J$. By Lemma~\ref{l: almost independent}, we may assume that the underlying sets $I'$ and $J'$ of $I$ and $J$ are almost independent sets of cardinality $d+2$. These sets are also almost independent sets for $\mathrm{AGL}_d(2)$ and hence, by Lemma~\ref{l:ai}, we have that $d$ is even and $I=(P_0,\ldots,P_d,P_{d+1})$ and $J=(Q_0,\ldots,Q_d,Q_{d+1})$, where $\{P_0,\ldots,P_d\}$ and $\{Q_0,\ldots,Q_d\}$ are frames and $P_{d+1}=\sum_iP_i$ and $Q_{d+1}=\sum_iQ_i$. Since $I\sim_{d+1}J$, we may assume $P_i=Q_i,$ for every $i\in \{0,\ldots,d\}$. Over the field with two elements, this implies $P_{d+1}=Q_{d+1}$ and $I=J$, contrary to the fact that $I\not\sim_{d+2}J$. 
\end{proof}

\begin{proof}[Proof of Theorem~$\ref{thrm:main}$]
The formula for the height follows from Lemma~\ref{l:0}. Indeed, a frame of an affine subspace of dimension $\ell$ has cardinality $\ell+1$, and the largest affine subspaces of $V$ have dimension $d$. The exceptional case $(d,q)=(1,2)$ also follows from Lemma~\ref{l:0}.

It follows from~\eqref{eq:2.1} that
$\rc(\mathrm{AGL}_d(q))\le d+2$
when $(d,q)\ne(1,2)$. When $(d,q)=(1,2)$, we have
$\mathrm{AGL}_1(2)=\mathrm{Sym}(2)$,
and hence
$\rc(\mathrm{AGL}_1(2))=2$.
Similarly, when $(d,q)=(1,3)$, we have
$\mathrm{AGL}_1(3)=\mathrm{Sym}(3)$,
and hence 
$\rc(\mathrm{AGL}_1(3))=2$.

Assume now that $q>2$ and $(d,q)\ne(1,3)$. For
$a=(a_1,\ldots,a_d)\in(\mathbb{F}_q^\ast)^d$,
define
$$
T(a)=
\left(0,e_1,\ldots,e_d,\sum_{i=1}^d a_i e_i\right).
$$
Suppose that $a,b\in(\mathbb{F}_q^\ast)^d$ satisfy
\begin{equation}\label{eq:q+2}
\sum_{i=1}^d a_i\ne 1,
\qquad
\sum_{i=1}^d b_i\ne 1,
\qquad
a\ne b.
\end{equation}
We claim that
\[
T(a)\sim_{d+1}T(b) \textrm{  and  }
T(a)\not\sim_{d+2}T(b). \]

To prove the claim, observe first that, after deleting the last entry, the two tuples are identical. After deleting $e_i$, the remaining points form a frame for $V$ in each tuple because $a_i$ and $b_i$ are non-zero. After deleting $0$, the remaining points form a frame for $V$ because
$\sum_{i=1}^d a_i\ne 1$ and $\sum_{i=1}^d b_i\ne 1$.
Thus, since $\mathrm{AGL}_d(q)$ is transitive on ordered frames we deduce that $T(a)\sim_{d+1}T(b)$.
On the other hand, an element mapping $T(a)$ to $T(b)$ would fix
$0,e_1,\ldots,e_d$
and would therefore be the identity. This is impossible because $a\ne b$. Hence
$T(a)\not\sim_{d+2}T(b)$, as claimed. 

The claim implies that, so long as there exist $a,b$ satisfying~\eqref{eq:q+2}, we have $\rc(\mathrm{AGL}_d(q))=d+2$. It remains to show that such $a$ and $b$ exist.

Suppose first that $q>3$. Choose $\alpha\in\mathbb{F}_q^\ast$ such that
$d\alpha\ne 1$.
Then choose
$\lambda\notin
\{0,\alpha,1-(d-1)\alpha\}$.
Such choices are possible because $q>3$. We may now take
$a=(\alpha,\ldots,\alpha)$
and $b=(\lambda,\alpha,\ldots,\alpha)$.

It remains to consider $q=3$ and $d>1$. If $d\equiv2\pmod 3$, take
$a=(1,\ldots,1)$
and $b=(-1,1,\ldots,1)$.
If $d\equiv0\pmod 3$, take 
$a=(-1,\ldots,-1)$
and
$b=(1,-1,\ldots,-1)$.
Finally, if $d\equiv1\pmod 3$ and $d\ge4$, take
$a=(-1,\ldots,-1)$
and
$b=(1,1,-1,\ldots,-1)$.
In each case, all coordinates are non-zero, the coordinate sums are not equal to $1$, and $a\ne b$. Therefore,
$\rc(\mathrm{AGL}_d(q))=d+2$
whenever $q>2$ and $(d,q)\ne(1,3)$.

Assume now that $q=2$ and $d>1$. 
Suppose first that $d$ is odd and $d>1$. Consider the two $(d+1)$-tuples
\begin{align*}
I&=(0,e_1,\ldots,e_d),\\
J&=\left(0,e_1,\ldots,e_{d-1},
\sum_{i=1}^{d-1}e_i\right).
\end{align*}
$I\not\sim_{d+1}J$.
We claim that 
$I\sim_d J$.
After deleting the last entry, the two tuples are identical. After deleting any $e_i$ with $1\le i\le d-1$, the remaining points form frames for affine subspaces of dimension $d-1$. Finally, after deleting $0$, the remaining points again form frames for affine subspaces of dimension $d-1$, because $d-1$ is even. Hence
$I\sim_d J$.
It follows that
$\rc(\mathrm{AGL}_d(2))=d+1$
when $d>1$ is odd by Lemma~\ref{l:binarysubgroups}.

Finally, suppose that $d$ is even. We first prove that
$\rc(\mathrm{AGL}_d(2))\ge d$.
When $d=2$, this is vacuous, so suppose that $d\ge4$. Consider the two $d$-tuples
\begin{align*}
I&=(0,e_1,\ldots,e_{d-1}),\\
J&=\left(0,e_1,\ldots,e_{d-2},
\sum_{i=1}^{d-2}e_i\right).
\end{align*}
The entries of $I$ form a frame for an affine subspace of dimension $d-1$, whereas the entries of $J$ do not. Thus
$I\not\sim_d J$.
The same deletion argument used in the odd-dimensional case shows that
$I\sim_{d-1}J$,
because $d-2$ is even. Therefore,
$\rc(\mathrm{AGL}_d(2))\ge d$.

It remains to prove the reverse inequality. Let $I,J\in V^{d+1}$ satisfy
$I\sim_d J$.
By Lemmas~\ref{l: almost independent} and~\ref{l:binarysubgroups}, it is enough to consider the case in which the sets of entries of $I$ and $J$ are almost independent and have cardinality $d+1$.
Since $d$ is even, Lemma~\ref{l:ai} shows that both sets of entries are frames for $V$. The regular action of $\mathrm{AGL}_d(2)$ on ordered frames therefore gives
 $I\sim_{d+1}J$.
Hence $
\rc(\mathrm{AGL}_d(2))\le d$.

We finish with the only assumption on $(d,q)$ being that it is not in $\{(1,2),(1,3)\}$.  Let $K$ be a subgroup of $\mathrm{AGL}_d(q)$. Then, from~\eqref{eq:2.1} and from the monotonicity of the height function, $\rc(K)\le \H(K)+1\le \H(\mathrm{AGL}_d(q))+1=d+2$. By Lemma~\ref{l:binarysubgroups}, the stronger inequality $\rc(K)\le d+1$ holds when $q=2$. This concludes the proof of Theorem~\ref{thrm:main}.
\end{proof}

\section{Affine groups having maximal relational complexity}\label{sec:maximal}

Let $p$ be prime, let $f$ and $d$ be positive integers and let $q=p^f$. Let $G=V\rtimes G_0\le \mathrm{AGL}_d(q)$, where $G_0$ acts irreducibly on $V$, and suppose that $\rc(G)=d+2$. Here $G_0\le \mathrm{GL}_d(q)$ is the stabilizer in $G$ of the zero vector. From Theorem~\ref{thrm:main} we deduce that $q>2$. Throughout this section we fix all of the above assumptions.

In this section, we aim to classify all such groups $G$ and thereby prove Theorem~\ref{thrm:main2}.
First, however, we make an observation that we use repeatedly throughout the whole argument. Let $a$ be a non-zero $d$-dimensional row vector and let $b$ be a non-zero $d$-dimensional column vector. Then $ba$ is a $d\times d$-matrix of rank $1$. Thus, by setting $B=I_d+ba$, we have $\mathrm{rank}(B-I_d)=1$. In fact, $B$ is a transvection if and only if $ab=0$.

In what follows, whenever $B$ is linear transformation of $V$ such that $\mathrm{rank}(B-I_d)=1$, then we call the image space, $\mathrm{Im}(B-I_d)$, the \textit{\textbf{moved line}} of $B$. In this case where $B$ is a transvection this is often called the \textit{\textbf{centre}} of $B$.

Given a subfield $\mathbb{F}_{q_0}\leq\mathbb{F}_q$ and a subset $X\subseteq V$, we use $\langle X\rangle_{\mathbb{F}_{q_0}}$ to denote the span of $X$ over the subfield $\mathbb{F}_{q_0}$. The following lemma will be helpful in this section and beyond.

\begin{lemma}\label{l:irredN}
    Let $\mathbb{F}_{q_0}\leq \mathbb{F}_q$ be an inclusion of fields. Let $E\leq H$ be the group of all matrices of the form $I_d+w^\top\delta$ where $w$ ranges over $W:=\langle X\rangle_{\mathbb{F}_{q_0}}$ for some $X\subseteq V$ and $\delta$ is a fixed non-zero row vector. Suppose that the common fixed-point space $C_V(E)=\mathbb{F}_q\delta$, and that $H\geq E$ is irreducible on $V$. Then $N:=\langle E^g\mid g\in H\rangle$ is irreducible on $V$.
\end{lemma}
\begin{proof}
    Let $0\ne U\le V$ be an $N$-invariant subspace. In particular, $U$ is $E$-invariant. Suppose first that there exist $x\in U$ and $u\in W$ such that $xu^\top\ne 0$. Then
$$
x(I_d+u^\top\delta)-x=(xu^\top)\delta\in U,
$$
and hence
$$
\mathbb{F}_q\delta\le U.
$$
On the other hand, if $xu^\top=0$ for every $x\in U$ and every $u\in W$, then
$$
U\le {\bf C}_V(E)=\mathbb{F}_q\delta.
$$
Since $U\ne 0$, we again obtain $
\mathbb{F}_q\delta\le U$.
Since $U$ is $N$-invariant, it is invariant under $E^g$ for every $g\in H$. Applying the preceding argument to each $E^g$, we obtain
$\mathbb{F}_q\delta^g\le U,$
for every $g\in H$. Since $H$ acts irreducibly on $V$ and $\delta\ne 0$, we have
$$
V=\sum_{g\in H}\mathbb{F}_q\delta^g.
$$
Therefore, $U=V$, and $N$ is irreducible.
\end{proof}

We proceed now with the one-dimensional case for Theorem~\ref{thrm:main2}. Thus $d=1$ and $G_0\le\mathbb{F}_q^\ast$. The reverse implication of the coming lemma is simply a special case of~\cite[Lemma 1.8.1]{GLieSp}, however we present an alternative proof which is perhaps more natural in our current setup.
\begin{lemma}\label{l:d=1}
    Suppose that $d=1$. Then $\rc(G)=d+2=3$ if and only if $|G_0|>2$.
\end{lemma}

\begin{proof}
 Suppose first that $|G_0|>2$. Choose $
\gamma\in G_0\setminus\{1,-1\}$,
and set
$$
\alpha=\frac{1}{1+\gamma},
\qquad
\beta=\frac{\gamma}{1+\gamma}.
$$
Consider the triples
$I=(0,1,\alpha)$ and 
$J=(0,1,\beta)$.
The pairs obtained by deleting the third entries are identical. The map
$$
x\longmapsto \gamma x
$$
maps $(0,\alpha)$ to $(0,\beta)$, while the map
$$
x\longmapsto \gamma^{-1}x+1-\gamma^{-1}
$$
maps $(1,\alpha)$ to $(1,\beta)$. Since $\gamma,\gamma^{-1}\in G_0$, it follows that
$I\sim_2J$.
On the other hand, $\alpha\ne\beta$, because $\gamma\ne1$, and the only affine transformation fixing both $0$ and $1$ is the identity. Hence
$I\not\sim_3J$.
Therefore, $\rc(G)\ge3$. Since $\H(G)\le \H(\mathrm{AGL}_1(q))=2$ by Theorem~\ref{thrm:main},
equation~\eqref{eq:2.1} gives $\rc(G)\le3$. Thus $\rc(G)=3$.

Conversely, suppose that $|G_0|\le2$. If $|G_0|=1$ then $G$ is regular, hence $\rc(G)=2$. Suppose that $|G_0|=2$. Then $2\mid (q-1)$, so $q$ is odd and we have
$G_0=\{1,-1\}$.
Let $$I=(x_1,\ldots,x_n)\sim_2(y_1,\ldots,y_n)=J$$
be $2$-equivalent tuples. If all entries of $I$ are equal, then all entries of $J$ are equal, and a translation maps $I$ to $J$. Otherwise, choose $r,s$ such that $x_r\ne x_s$. After applying an element of $G$ to $J$, we may assume that
$y_r=x_r$ and $y_s=x_s$.
For every $k\leq n$, since $I\sim_2 J$ we deduce that there is some $(u,\ell)\in\mathbb{F}_q\rtimes\{1,-1\}$ such that $$(u+\ell x_r,u+\ell x_k)=(x_r,x_k)^{(u,\ell)}=(y_r,y_k)=(x_r,y_k).$$ In particular, $y_k-x_r=u+\ell x_k-u-\ell x_r=\ell(x_k-x_r)$, that is, for every $k\leq n$, $y_k-x_r=\pm(x_k-x_r)$ and similarly $y_k-x_s=\pm(x_k-x_s)$.
If $y_k\ne x_k$, the first equality gives $y_k=2x_r-x_k$,
while the second gives $y_k=2x_s-x_k$.
Since the characteristic is odd, this implies $x_r=x_s$, a contradiction. Therefore, $y_k=x_k$ for every $k$, and hence $I\sim_nJ$. Thus $\rc(G)=2$. 
\end{proof}
Lemma~\ref{l:d=1} proves the one-dimensional case of Theorem~\ref{thrm:main2}, namely part~\eqref{part1}.

We assume then for the remainder of the section that $d>1$. As $\rc(G)=d+2$, there exist two $(d+2)$-tuples $I$ and $J$ such that $I\sim_{d+1}J$ and $I\not\sim_{d+2}J$.
By Lemma~\ref{l: almost independent}, the underlying sets of $I$ and $J$ are almost independent sets of cardinality $d+2$. After simultaneously permuting the coordinates, replacing $J$ by a $G$-conjugate, and conjugating $G$, $I$, and $J$ simultaneously by a suitable element of $\mathrm{AGL}_d(q)$, Lemma~\ref{l:aiodd} allows us to assume that
\begin{equation*}
I=\left(0,e_1,\ldots,e_d,\sum_{i=1}^d\alpha_i e_i\right),\,
J=\left(0,e_1,\ldots,e_d,\sum_{i=1}^d\beta_i e_i\right),
\end{equation*}
where
\begin{equation}\label{eq:condition1}
\alpha_i\ne 0,
\qquad
\beta_i\ne 0
\end{equation}
for every $i\in\{1,\ldots,d\}$, and
\begin{equation}\label{eq:condition2}
\sum_{i=1}^d\alpha_i\ne 1,
\qquad
\sum_{i=1}^d\beta_i\ne 1.
\end{equation}
Set
$$
v=\sum_{i=1}^d\alpha_i e_i,
\qquad
w=\sum_{i=1}^d\beta_i e_i,
\qquad
\delta=w-v.
$$
Since $I\not\sim_{d+2}J$, we have $v\ne w$, and hence $\delta\ne 0$.

For each $i\in\{1,\ldots,d\}$, choose $(v_i,B_i)\in G$ mapping
$$
(0,e_1,\ldots,e_{i-1},e_{i+1},\ldots,e_d,v)
$$
to
$$
(0,e_1,\ldots,e_{i-1},e_{i+1},\ldots,e_d,w).
$$
Since this affine transformation fixes $0$, we have $v_i=0$. Also choose $(v_0,B_0)\in G$ mapping
$$
(e_1,\ldots,e_d,v)
$$
to
$$
(e_1,\ldots,e_d,w).
$$
The linear parts $B_0,B_1,\ldots,B_d$ belong to $G_0$. 

For each $i\in\{1,\ldots,d\}$, define
\begin{equation}\label{def:ci}
c_i=\frac{1}{\alpha_i}e_i^\top ,
\end{equation}
where $e_i^\top $ denotes the transpose of the row vector $e_i$.
\begin{lemma}\label{l:Bidef}
    Let $i\in \{1,\dots,d\}$, and let $\delta$, $c_i$ and $B_i$ be defined as above. Then $$B_i=I_d+c_i\delta$$
\end{lemma}
\begin{proof}
    Since $(0,e_1,\dots,e_{i-1},e_{i+1},\dots,e_d,v)$ and $(0,e_1,\dots,e_{i-1},e_{i+1},\dots,e_d,w)$ are ordered frames, it suffices to show that $I_d+c_i\delta$ maps $v$ to $w$ and fixes $e_j$ for every $j\ne i$.

    Let $j\in\{1,\dots,d\}\setminus\{i\}$. Then $$e_j(I_d+c_i\delta)=e_j+\frac{1}{\alpha_i}e_je_i^\top \delta=e_j+0,$$ as desired. Moreover, $$v(I_d+c_i\delta)=v+\frac{1}{\alpha_i}\left(\sum_{j=1}^d\alpha_je_j\right)e_i^\top \delta=v+\delta=w,$$ as was to be shown.
\end{proof}
Set
$$
\alpha=\sum_{i=1}^d\alpha_i,
\qquad
\beta=\sum_{i=1}^d\beta_i,
$$
and let
$$
c_0=-\frac{1}{1-\alpha}\mathbf{1}^\top ,
$$
where $\mathbf{1}=(1,\ldots,1)\in V$.
\begin{lemma}\label{l:B0def}
    Take $\alpha, \beta,c_0,v_0$ and $B_0$ as defined above. Then \begin{equation*}\label{eq:B0}
B_0=I_d+c_0\delta
\quad\text{and}\quad
v_0=\frac{1}{1-\alpha}\delta.
\end{equation*}
\end{lemma}
\begin{proof}
    As in the proof of Lemma~\ref{l:Bidef}, it suffices to show that $\frac{1}{1-\alpha}\delta+e_i(I_d+c_0\delta)=e_i$ for each $i\in\{1,\dots,d\}$ and that $\frac{1}{1-\alpha}\delta+v(I_d+c_0\delta)=w$.

    Let $i\in\{1,\dots,d\}$. Then $(e_ic_0)\delta=-\frac{1}{1-\alpha}\delta$, whence $\frac{1}{1-\alpha}\delta+e_i(I_d+c_0\delta)=e_i$. Additionally, $vc_0=-\frac{\alpha}{1-\alpha}$, so $$\frac{1}{1-\alpha}\delta+v(I_d+c_0\delta)=\frac{1}{1-\alpha}\delta+v-\frac{\alpha}{1-\alpha}\delta=\delta+v=w,$$ as was to be shown.
\end{proof}

For each $i\in\{0,\ldots,d\}$, from Lemmas~\ref{l:Bidef} and~\ref{l:B0def}, the matrix $B_i-I_d$ has rank $1$ and
$\lambda_i:=\det(B_i)=1+\delta c_i$ by the matrix determinant lemma. Thus
\begin{equation}\label{eq:det1}
\lambda_i=
\begin{cases}
\dfrac{\beta_i}{\alpha_i}
& \textrm{if }i\in\{1,\ldots,d\},\\[2mm]
\dfrac{1-\beta}{1-\alpha}
& \textrm{if }i=0.
\end{cases}
\end{equation}

The description above shows that every $B_i$ preserves the one-dimensional subspace $\mathbb{F}_q\delta$ and acts trivially on the quotient $V/\mathbb{F}_q\delta$. Thus, the same is true for
$$
K:=\langle B_0,B_1,\ldots,B_d\rangle\le G_0.
$$ Since $G_0$ is irreducible and $d>1$, the subspace $\mathbb{F}_q\delta$ is not $G_0$-invariant. Hence
$K\ne G_0$.

Let
$$
\mathbb{F}_{q_0}=\mathbb{F}_p(\lambda_1,\ldots,\lambda_d)\le \mathbb{F}_q,
$$
where $q_0=p^{f_0}$ for some divisor $f_0$ of $f$.
Since $v\ne w$, there exists $r\in\{1,\ldots,d\}$ such that $\alpha_r\ne\beta_r$, and hence 
\begin{equation}\label{lambdar}
\lambda_r\ne 1.\end{equation} Fix such an $r$ and for each $j\in\{1,\dots,d\}$ define
$$
u_j=(\lambda_j-1)c_r-(\lambda_r-1)c_j.
$$ and set $$
W=
\left\langle
u_j\mid j\in\{1,\ldots,d\}\setminus\{r\}
\right\rangle_{\mathbb{F}_{q_0}}.
$$

\begin{lemma}
    The group $K$ contains an elementary abelian group $E=\{I_d+u\delta\mid u\in W\}$ of transvections of dimension $d-1$ over $\mathbb{F}_{q_0}$. Moreover, the common fixed-point space of $E$ on $V$ is $\mathbf{C}_V(E)=\mathbb{F}_q\delta$. Set $N:=\langle E^g\mid g\in G_0\rangle$. Then $N\trianglelefteq G_0$, $N\leq \mathrm{SL}_d(q)$, and $N$ is irreducible.
\end{lemma}

\begin{proof}
Throughout this proof we will repeatedly use that $\lambda_j-1=\delta c_j$ for $j\in\{0,\dots, d\}$. This implies first that
\begin{equation}\label{eq:key}
\delta u_j=(\lambda_j-1)\delta c_r -(\lambda_r-1)\delta c_j=(\delta c_j)\delta c_r -(\delta c_r)\delta c_j=0.
\end{equation}
Moreover, for any $j\in\{1,\dots,d\}$, 
\[B_j^{-1}=(I_d+c_j\delta)^{-1}=I_d-\dfrac{c_j\delta}{\lambda_j}.\]
Therefore, for $j\in\{1,\ldots,d\}\setminus\{r\}$,
\[
(B_jB_r)^{-1}=\left(I_d-\dfrac{c_r\delta}{\lambda_r}\right)\left(I_d-\dfrac{c_j\delta}{\lambda_j}\right)=I_d-\frac{c_r\delta}{\lambda_r}-\frac{c_j\delta}{\lambda_j}+\frac{c_r\delta c_j\delta}{\lambda_r\lambda_j}= I_d-\frac{c_j\delta}{\lambda_j}-\frac{c_r\delta}{\lambda_r\lambda_j},\]
where the final equality is using (again) that $\delta c_j=\lambda_j-1$. Similarly, \[B_rB_j=(I_d+c_r\delta)(I_d+c_j\delta)=I_d+c_r\delta+c_j\delta+(\lambda_j-1)c_r\delta=I_d+\lambda_jc_r\delta+c_j\delta, 
\]
whence
\begin{align*}(B_jB_r)^{-1}B_rB_j&=\left(I_d-\frac{c_j\delta}{\lambda_j}-\frac{c_r\delta}{\lambda_r\lambda_j}\right)(I_d+\lambda_jc_r\delta+c_j\delta)\\&=I_d-\frac{c_j\delta}{\lambda_j}-\frac{c_r\delta}{\lambda_r\lambda_j}+\lambda_jc_r\delta-c_j\delta c_r\delta-\frac{c_r\delta c_r\delta}{\lambda_r}+c_j\delta-\frac{c_j\delta c_j\delta}{\lambda_j}-\frac{c_r\delta c_j\delta}{\lambda_r\lambda_j}\\&=
I_d-c_r\delta+\lambda_jc_r\delta-(\lambda_r-1)c_j\delta\\&=I_d+u_j\delta.\end{align*}

By~\eqref{eq:key},
$
(u_j\delta)^2=0,
$
thus, for $j\ne r$ the element $I_d+u_j\delta$ is a transvection in $K$.

The vectors $u_j$, with $j\in\{1,\ldots,d\}\setminus\{r\}$, are linearly independent over $\mathbb{F}_{q_0}$. Indeed, the vectors $c_1,\ldots,c_d$ form an $\mathbb{F}_q$-basis of the column-vector space, and, in any linear combination of the $u_j$, the coefficient of $c_j$ is $-(\lambda_r-1)$ times the coefficient of $u_j$.

If $u$ is a column vector satisfying $\delta u=0$, then 
$$
B_i^{-1}(I_d+u\delta)B_i=(I_d-\lambda_i^{-1}c_i\delta)(I_d+u\delta)(I_d+c_i\delta)=(I_d+u\delta-\lambda_i^{-1}c_i\delta)(I_d+c_i\delta)=I_d+\lambda_i u\delta.
$$
Also, if $\delta u=\delta u'=0$, then
$$
(I_d+u\delta)(I_d+u'\delta)=I_d+(u+u')\delta.
$$
It follows that $K$ contains $E$. Moreover, $\dim_{\mathbb{F}_{q_0}}(W)=d-1$, $|E|=q_0^{d-1}$,
and every nonidentity element of $E$ is a transvection in $K$, and hence in $G_0$: products add the parameters $u$, while conjugation by $B_i$
multiplies them by $\lambda_i$; thus the admissible scalars contain $\mathbb{F}_{q_0}$. We have now proved the first claim.

Observe next that
$$
\mathbb{F}_qW
=
\ker(\delta),
$$
where by abuse of notation we are identifying with $\delta$ the linear mapping $V^\top \to \mathbb{F}_q$ sending $u$ to $\delta u$. Given vectors $u,v\in V$, $v$ is fixed by $I+u^\top\delta$ if and only if $vu^\top=0$, whence the common fixed-point space of $E$ is (isomorphic to) the annihilator of $\mathbb{F}_q W=\ker(\delta)$, and hence has dimension 1 over $\mathbb{F}_q$. Therefore, 
$$
{\bf C}_V(E)=\mathbb{F}_q\delta.
$$

Clearly $N=\langle E^g\mid g\in G_0\rangle\unlhd G_0$ and $N\le \mathrm{SL}_d(q)$---that $N$ is irreducible follows from Lemma~\ref{l:irredN}.
\end{proof}

In summary, $N$ is an irreducible subgroup of $\mathrm{SL}_d(q)$ generated by transvections. 

Choose a subfield $\mathbb{F}_{q_1}\le\mathbb{F}_q$ of smallest
possible order such that, after conjugation in $\mathrm{GL}_d(q)$,
$$
N\le \mathrm{GL}_d(q_1).
$$
Replacing $N$, $G_0$, and $E$ by their corresponding conjugates, we may assume that
this containment holds. Since $N\le\mathrm{SL}_d(q)$, we have $
N\le\mathrm{SL}_d(q_1)$.

Let
$$
V_1=\mathbb{F}_{q_1}^d.
$$
The group $N$ acts irreducibly on $V_1$. Indeed, if
$0<U_1<V_1$ were an $N$-invariant $\mathbb{F}_{q_1}$-subspace, then $U_1\otimes_{\mathbb{F}_{q_1}}\mathbb{F}_q$
would be a proper non-zero $N$-invariant $\mathbb{F}_q$-subspace of
$V$, contrary to the irreducibility of $N$ on $V$.

Moreover, by the minimality of $q_1$, the group $N$ is not contained
in a subfield subgroup of $\mathrm{SL}_d(q_1)$. We may therefore apply
\cite[Theorem~1.4]{Eberhard} (Eberhard attributes this result to Kantor~\cite{Kantor}, although the version in~\cite{Eberhard} is more tailored to our application). Since $N$ is irreducible and generated
by transvections, one of the following holds, up to conjugacy:
\begin{enumerate}
\item\label{case1} $N=\mathrm{SL}_d(q_1)$;
\item\label{case2} $q_1$ is a square and
$N=\mathrm{SU}_d(\sqrt{q_1})$;
\item\label{case3} $d$ is even and
$N=\mathrm{Sp}_d(q_1)$;
\item\label{case4} $d$ and $q_1$ are even and
$N=\mathrm{O}_d^\pm(q_1)$;
\item\label{case5} $q_1$ is even and
$N\cong C_a^{d-1}\rtimes\mathrm{Sym}(d)$,
where $a>1$ is odd and $q_1$ is the smallest power of $2$ such
that $a$ divides $q_1-1$;
\item\label{case6} $q_1=2$, $d$ is even, and
$N$ is the natural image of $\mathrm{Sym}(d+1)$
or $\mathrm{Sym}(d+2)$ on its fully deleted permutation module;
\item\label{case7} $d=2$, $q_1=9$, and
$N\cong\mathrm{SL}_2(5)$;
\item\label{case8} $d=3$, $q_1=4$, and
$N\cong 3\mathord{\cdot}A_6$;
\item\label{case9} $d=6$, $q_1=4$, and
$N\cong
3\mathord{\cdot}\mathrm{P}\Omega_{6}^{-,\pi}(3)$;
\item\label{case10} $d=2$, $q_1$ is even, and $N$ is dihedral of order dividing $2(q_1+1)$.
\end{enumerate}
\vspace{11pt}
\noindent{\sc Remark.} Case~\eqref{case10} is inadvertently left out of~\cite[Theorem 1.4]{Eberhard} and the statement of~\cite[Theorem II]{Kantor}, although it can be found in~\cite[Section 11]{Kantor} and~\cite{Dickson}.\vspace{11pt}

For the definition of $\Omega_6^{-,\pi}(3)$, see~\cite[page~349]{Kantor}.

We deal with each of these cases in turn and we  use the particular elementary abelian subgroup $E$ constructed above. Recall that $q>2$, that every nonidentity element of $E$ is a transvection, that all these transvections have the same moved line (i.e. $\mathrm{Im}(E-I_d)$ is the same for all transvections in $E$), and that their common fixed-point space is precisely that moved line.

These properties remain true after conjugation. More precisely, \begin{align}\begin{split}\label{eq:L1}\text{there
exists a one-dimensional subspace $L_1$ of $V_1$ such that}\\\text{
$[V_1,x]=L_1,$
for every $x\in E\setminus\{I_d\}$, and
${\bf C}_{V_1}(E)=L_1$.}\end{split}\end{align}
Indeed, scalar extension from $\mathbb{F}_{q_1}$ to $\mathbb{F}_q$
gives
$$
[V_1,x]\otimes_{\mathbb{F}_{q_1}}\mathbb{F}_q=[V,x]
$$
and
$$
{\bf C}_{V_1}(E)\otimes_{\mathbb{F}_{q_1}}\mathbb{F}_q
=
{\bf C}_V(E).
$$
The corresponding subspaces over $\mathbb{F}_q$ are both equal to the
one-dimensional subspace $\mathbb{F}_q\delta$, and hence the
corresponding subspaces over $\mathbb{F}_{q_1}$ are both equal to
$L_1$.

We first exclude the unitary (case~\eqref{case2}), symplectic (case~\eqref{case3}), and orthogonal (case~\eqref{case4}) cases when
$d\ge3$.
\begin{lemma}\label{l:noforms}
    Suppose that $d\geq3$. Then $N$ preserves no nondegenerate alternating, Hermitian, or quadratic form.
\end{lemma}
\begin{proof}
Suppose otherwise. If $x$ is a transvection in $E$ preserving such a form and $[V_1,x]=L_1$, then $
{\bf C}_{V_1}(x)=L_1^\perp$.
Since every nonidentity element of $E$ has moved line $L_1$, it follows
that
${\bf C}_{V_1}(E)=L_1^\perp$.
On the other hand, we have seen above that $
{\bf C}_{V_1}(E)=L_1$.
Thus $L_1=L_1^\perp$.
This is impossible when $d\ge3$, because 
$\dim_{\mathbb{F}_{q_1}}L_1=1$
whereas
$\dim_{\mathbb{F}_{q_1}}L_1^\perp=d-1>1$.
Therefore the unitary, symplectic, and orthogonal cases do not occur
when $d\ge3$.
\end{proof}

We next exclude the cases~\eqref{case5},~\eqref{case6},~\eqref{case8},~\eqref{case9}, and~\eqref{case10}.
\begin{lemma}\label{l:d=2}
    If $N\ne \mathrm{SL}_d(q_1)$, then $d=2$ and $N$ is either a classical group in characteristic $p$ or $q_1=9$ and $N\cong \mathrm{SL}_2(5)\leq \mathrm{GL}_2(9).$
\end{lemma}
\begin{proof}We consider first the monomial case~\eqref{case5}, that is
$$
N\cong C_a^{d-1}\rtimes\mathrm{Sym}(d)
$$
in its natural monomial representation. Since $a$ is odd, a
transvection in $N$ has permutation part consisting of a single
transposition. After permuting the basis, its non-trivial block has the
form
$$
\begin{pmatrix}
0&c\\
c^{-1}&0
\end{pmatrix}
$$
for some $c\in\mathbb{F}_{q_1}^\ast$. Its moved line is
$$
\left\langle e_i+ce_j\right\rangle.
$$
The moved line determines the pair $\{i,j\}$ and the scalar $c$, and
hence determines the transvection. Thus there is at most one
transvection in $N$ with any prescribed moved line.

All nonidentity elements of $E$ have the same moved line. Therefore
$E$ has at most one nonidentity element, and hence $|E|\le 2$.
This is impossible. Indeed, the monomial case occurs only in
characteristic $2$, and~\eqref{lambdar} implies that
$\mathbb{F}_{q_0}\ne\mathbb{F}_2$.
Consequently, $q_0\ge4$ and $|E|=q_0^{d-1}\ge4$.
Thus the monomial case cannot occur.

We now exclude case~\eqref{case6}. A group of this type preserves
a nondegenerate alternating form on its fully deleted permutation module;
see~\cite[Section~6.2]{Eberhard}. Hence, we may assume that $d=2$ by Lemma~\ref{l:noforms}. Since $d=2$, the image of the symmetric group on the fully deleted permutation
module is isomorphic to $\mathrm{SL}_2(2)$ and has Sylow $2$-subgroups
of order $2$. This is also impossible, since in characteristic $2$ we
have $|E|=q_0\ge4$. Similarly, a dihedral group of order dividing $2(q_1+1)$ has Sylow 2-subgroup of order $2$, so case~\eqref{case10} cannot occur.

Consider now case~\eqref{case8}.
This case has characteristic $2$. Since $q_0\ge4$ and $d=3$, we have
$|E|=q_0^2\ge4^2=16$.
However, a Sylow $2$-subgroup of $3.A_6$ has order $8$.
Therefore $3.A_6$ cannot contain $E$, and this case is
impossible.

Consider now case~\eqref{case9}.
Again the characteristic is $2$, and $|E|=q_0^5\ge4^5=2^{10}$.
On the other hand, a Sylow $2$-subgroup of $3.\mathrm{P}\Omega_{6}^{-,\pi}(3)$
has cardinality $2^8$.
Thus this group cannot contain $E$, and this case is also impossible.
\end{proof}

It remains to consider the low-dimensional classical cases and case~\eqref{case7}. We deal with the classical case first.
\begin{lemma}
    If $N$ is a classical group in characteristic $p$ then $N=\mathrm{SL}_d(q_1)$.
\end{lemma}
\begin{proof}
    By Lemma~\ref{l:noforms}, we can assume that $d=2$. Now, 
$\mathrm{Sp}_2(q_1)=\mathrm{SL}_2(q_1)$,
so the symplectic case is already the linear case. Moreover, if
$q_1=s^2$, then the natural group $\mathrm{SU}_2(s)$ is conjugate in
$\mathrm{GL}_2(q_1)$ to $\mathrm{SL}_2(s)$. It is therefore contained
in a proper subfield subgroup, contrary to the minimality of $q_1$.
The two-dimensional orthogonal case can occur only in characteristic
$2$. In that case, $
|\mathrm{O}_2^\pm(q_1)|=2(q_1\mp1)$,
so every Sylow $2$-subgroup has order $2$. Since $|E|\ge4$, this case
is also impossible.
\end{proof}

Consequently, in dimension $2$, the only possibilities are
$$
N=\mathrm{SL}_2(q_1)
$$
and the exceptional case~\eqref{case7}, that is, 
$q_1=9$ and $N\cong\mathrm{SL}_2(5)$.
\begin{lemma}
    Suppose that $N=\mathrm{SL}_d(q_1)$. Then
$\mathrm{SL}_d(q_1)\unlhd G_0
\le \mathrm{GL}_d(q_1)\mathbb{F}_q^\ast$, and\\
$\mathrm{SL}_d(q_1)<G_0\cap \mathrm{GL}_d(q_1),$ where $\mathbb{F}_q^\ast$ is identified with the group of scalar matrices.
\end{lemma}
\begin{proof}After conjugation in $\mathrm{GL}_d(q)$,
$
N=\mathrm{SL}_d(q_1)\unlhd G_0.
$
Therefore
$$
G_0\le
{\bf N}_{\mathrm{GL}_d(q)}
\bigl(\mathrm{SL}_d(q_1)\bigr)=\mathrm{GL}_d(q_1)\mathbb{F}_q^\ast.
$$
Hence
$
\mathrm{SL}_d(q_1)\unlhd G_0
\le
\mathrm{GL}_d(q_1)\mathbb{F}_q^\ast$.

Next, we may write $B_r=\lambda B$ for some $\lambda\in\mathbb{F}_q^\ast$ and $B\in\mathrm{GL}_d(q_1)$. Since each of the $B_i$ have the same moved line $L$ over $V$, it follows that $B$ preserves $L_1$ as defined in~\eqref{eq:L1}. Additionally, we have already seen that $B_r$ acts trivially on $V/L$, whence $B$ acts as the scalar $\lambda^{-1}$ on $V/L$ and hence on $V_1/L_1$. Thus, $\lambda\in\mathbb{F}_{q_1}^\ast$, and so $B_r\in\mathrm{GL}_d(q_1)$. Moreover, $\det(B_r)=\lambda_r\ne1$, so $\mathrm{SL}_d(q_1)<G_0\cap \mathrm{GL}_d(q_1)$.\end{proof}
Thus part~\eqref{part2} of Theorem~\ref{thrm:main2} is satisfied.

For the remainder of the section we consider the exceptional case~\eqref{case7}. Thus, we assume that
$p=3$, $d=2$, $q_1=9$ and that
$N\cong\mathrm{SL}_2(5)$.
In particular, $\mathbb{F}_9\le\mathbb{F}_q$, and hence $f$ is even.
Let $\pi:\mathrm{GL}_2(q)\longrightarrow\mathrm{PGL}_2(q)$
be the natural projection. Set
$\overline{N}=\pi(N)$.
Since the center of $N\cong\mathrm{SL}_2(5)$ consists of the scalar
matrices $\{\pm I_2\}$, we have
$\overline{N}\cong\mathrm{PSL}_2(5)\cong \mathrm{Alt}(5)$.
\begin{lemma}\label{eq:normalizer}
    Fix the notation of above. Then \begin{equation*}
{\bf N}_{\mathrm{PGL}_2(q)}(\overline{N})=\overline{N}.\end{equation*}
\end{lemma}
\begin{proof}
Let $\overline{g}\in\mathrm{PGL}_2(q)$ centralize $\overline{N}$, and
choose $g\in\mathrm{GL}_2(q)$ with $\pi(g)=\overline{g}$. For every
$n\in N$, there exists $\lambda_n\in\mathbb{F}_q^\ast$ such that
$g^{-1}ng=\lambda_n n$.
The map
$n\longmapsto\lambda_n$
is a homomorphism from $N$ to $\mathbb{F}_q^\ast$. Since
$N\cong\mathrm{SL}_2(5)$ is perfect, this homomorphism is trivial.
Therefore, $g^{-1}ng=n$
for every $n\in N$.
Now, the natural two-dimensional representation of $N$ is absolutely
irreducible. Indeed, if it were reducible over the algebraic closure, then $N$ would be conjugate to a subgroup of the group of upper
triangular matrices, which is solvable. This is impossible because
$N\cong\mathrm{SL}_2(5)$ is not solvable. Hence, by Schur's lemma,
$${\bf C}_{\mathrm{GL}_2(q)}(N)=\mathbb{F}_q^\ast.$$
Thus $g$ is scalar and $\overline{g}=1$, proving that
${\bf C}_{\mathrm{PGL}_2(q)}(\overline{N})=1$.
Conjugation now gives an embedding
$$
{\bf N}_{\mathrm{PGL}_2(q)}(\overline{N})
\longrightarrow
\mathrm{Aut}(\overline{N})\cong \mathrm{Sym}(5).
$$
Since the image contains the group of inner automorphisms
$\mathrm{Inn}(\overline{N})\cong \mathrm{Alt}(5)$,
we have
${\bf N}_{\mathrm{PGL}_2(q)}(\overline{N})
\cong \mathrm{Alt}(5)$ or ${\bf N}_{\mathrm{PGL}_2(q)}(\overline{N})
\cong \mathrm{Sym}(5)$. If the first possibility occurs, then $\bar N={\bf N}_{\mathrm{PGL}_2(q)}(\overline N)$ and our claim is proven. Thus
suppose, for a contradiction, that
${\bf N}_{\mathrm{PGL}_2(q)}(\overline{N})\cong \mathrm{Sym}(5)$.
Choose a subfield
$\mathbb{F}_{3^m}\le\mathbb{F}_q$
of smallest possible order such that, after conjugation,
$$
{\bf N}_{\mathrm{PGL}_2(q)}(\overline{N})
\le\mathrm{PGL}_2(3^m).
$$
Since this group contains $\overline{N}\cong \mathrm{Alt}(5)$, we have $m\ge2$.
Let $M$ be a maximal subgroup of $\mathrm{PGL}_2(3^m)$ containing
${\bf N}_{\mathrm{PGL}_2(q)}(\overline{N})$.  By Dickson's classification~\cite{Giudici},
$\mathrm{PSL}_2(3^m)$ has no subgroup isomorphic to
$\mathrm{Sym}(5)$. Hence
${\bf N}_{\mathrm{PGL}_2(q)}(\overline{N})
\not\le \mathrm{PSL}_2(3^m)$, and therefore $M$ does not contain
$\mathrm{PSL}_2(3^m)$.

By~\cite[Theorem~3.5]{Giudici} (see also~\cite{Dickson}), $M$ is one of the following:
\begin{enumerate}
\item a point stabilizer;
\item a dihedral group;
\item a subgroup isomorphic to $S_4$;
\item a proper subfield subgroup $\mathrm{PGL}_2(3^{m_0})$.
\end{enumerate}
The first two possibilities are solvable, and the third has order $24$,
so none of them can contain a subgroup isomorphic to $\mathrm{Sym}(5)$. The final
possibility contradicts the minimality of $m$. Therefore,
${\bf N}_{\mathrm{PGL}_2(q)}(\overline{N})
\not\cong \mathrm{Sym}(5).$ Hence the result.    
\end{proof}

\begin{lemma}
    Suppose that $p=3$ and that $N=\mathrm{SL}_2(5)<\mathrm{SL}_2(9)$. Then $G_0=NZ$ where $Z\leq\mathbb{F}_q^\ast$, and $G_0$ contains a scalar matrix of order 4.
\end{lemma} 
\begin{proof}First, observe that Lemma~\ref{eq:normalizer} implies that $
{\bf N}_{\mathrm{GL}_2(q)}(N)=N\mathbb{F}_q^\ast$.
Set $Z=G_0\cap \mathbb{F}_q^\ast$. Since $N\unlhd G_0$, we have
$G_0=ZN$. 

It remains to show that $G_0$ contains a scalar matrix of order $4$. Since $K\leq G_0=ZN$ we have that
$B_0,B_1,B_2$ are three distinct non-scalar elements of $G_0$, all having the
same moved line $L=\mathbb{F}_q\delta$.

Let $\overline{B}_i$ denote the projective image of $B_i$. Since
$G_0/Z\cong \mathrm{Alt}(5)$,
the order of each $\overline{B}_i$ is $2$, $3$, or $5$.
We claim that at least one $\overline{B}_i$ has order $2$. Suppose
otherwise.
If some $\overline{B}_i$ has order 3, then $$B_i^3=(I_d+c_i\delta)^3=I_d+(c_i\delta)^3=I_d+(\delta c_i)^2c_i\delta\in\mathbb{F}_q^\ast,$$ since we are working in characteristic 3. But $c_i\delta$ has rank 1, and thus $\delta c_i=0$, that is, $B_i$ is a transvection. Thus, if all three projective images have order $3$, then each $B_i$ is a
transvection belonging to $N$. The transvections in $N$ with moved
line $L$, together with the identity, form a $3$-subgroup of $N$.
Since a Sylow $3$-subgroup of $N$ has order $3$,
there are at most two nonidentity transvections in $N$ with moved line
$L$. This contradicts the fact that $B_0,B_1,B_2$ are distinct.

Suppose that some $\overline{B}_i$ has order $3$ and some
$\overline{B}_j$ has order $5$. Then $
\langle\overline{B}_i,\overline{B}_j\rangle$
fixes $L$ and is therefore a proper subgroup of $\mathrm{Alt}(5)$. However, no
proper subgroup of $\mathrm{Alt}(5)$ contains elements of both orders $3$ and $5$,
a contradiction.

Finally, suppose that all three projective images have order $5$.
Since they generate a proper subgroup of $\mathrm{Alt}(5)$ fixing $L$, they belong
to the same cyclic subgroup of order $5$. Consequently, $B_0,B_1$ and $B_2$  all fix setwise the same two lines.  It follows that the rank-one
matrices $
B_i-I_2=c_i\delta$
have the same kernel, and hence the vectors $c_0,c_1,c_2$ are proportional.
This contradicts the linear independence of $c_1$ and $c_2$ -- see~\eqref{def:ci}.

Thus some $\overline{B}_i$ has order $2$. Write $
B_i=zn$,
with $zI_2\in Z$ and $n\in N$. Since $B_i$ has an eigenvalue $1$ and its
projective image has order $2$, its other eigenvalue is $-1$. Therefore,
$\det(B_i)=-1$.
On the other hand,  $\det(B_i)=z^2$,
because $n\in\mathrm{SL}_2(q)$. Hence 
$z^2=-1$.
Thus $zI_2$ has order $4$ and belongs to $G_0$.
\end{proof}
Thus part~\eqref{part3} of Theorem~\ref{thrm:main2} is now proved. It remains to show the converse of Theorem~\ref{thrm:main2}.

\section{The converse in Theorem~\ref{thrm:main2}}\label{sec:converse}
In this section, we address the reverse implication in Theorem~\ref{thrm:main2}. When part~\eqref{part1} holds, that is, $d=1$ and $|G_0|>2$, we have shown in Lemma~\ref{l:d=1} that $\rc(V\rtimes G_0)=3=d+2$. Therefore, we may suppose that $d>1$. We first deal with part~\eqref{part2}.

\begin{lemma}
     Suppose that $d\ge 2$ and that there exists a divisor $f_1$ of $f$ such that $\mathrm{SL}_d(p^{f_1})\unlhd G_0
\le \mathrm{GL}_d(p^{f_1})\mathbb{F}_q^\ast$, and
$\mathrm{SL}_d(p^{f_1})<G_0\cap \mathrm{GL}_d(p^{f_1})$. Then $\rc(G)=d+2$.

\end{lemma}

\begin{proof}
Set $q_1=p^{f_1}$, let  $H_1=G_0\cap\mathrm{GL}_d(q_1)$ 
and  $D=\det(H_1)\leq\mathbb{F}_{q_1}^\ast$.
Since
$$
\ker\left(
\det:\mathrm{GL}_d(q_1)\longrightarrow
\mathbb{F}_{q_1}^\ast
\right)
=
\mathrm{SL}_d(q_1)
$$
and $\mathrm{SL}_d(q_1)<H_1$,
we have $H_1 =\left\{
A\in\mathrm{GL}_d(q_1):
\det(A)\in D
\right\}$.The strict inclusion $\mathrm{SL}_d(q_1)<H_1$
implies that $D\neq\{1\}$ and, in particular, $q_1>2$. Choose 
$\lambda\in D$ with
$\lambda\neq1$. Next, choose non-zero $\gamma_2,\ldots,\gamma_d\in\mathbb{F}_{q_1}^\ast$ with $\sum_{i=2}^d\gamma_i=1$. The existence of $\gamma_2,\ldots,\gamma_d$ is guaranteed by the fact that $q_1>2$.
Define $z=\sum_{i=2}^d\gamma_i e_i$, $v=e_1+z$ and
$w=\lambda e_1+z$.
By construction, all coordinates  of $v$ and $w$ are non-zero, the sum of the coordinates is not $1$ and $v\neq w$. In particular,~\eqref{eq:condition1} and~\eqref{eq:condition2} are satisfied.

Let $\delta=w-v$, $I=(0,e_1,\ldots,e_d,v)$ and $J=(0,e_1,\ldots,e_d,w)$. Consider the matrices $B_1,\ldots,B_d$ in Lemma~\ref{l:Bidef} and the matrix $B_0$ in Lemma~\ref{l:B0def}.

From~\eqref{eq:det1}, for $0\leq i\leq d$, we have
$$
\det(B_i)=
\begin{cases}
\lambda&\textrm{when }i\in\{0,1\},\\
1&\textrm{when }i>1.
\end{cases}
$$
All entries of the matrices \(B_0,B_1,\ldots,B_d\) belong to
\(\mathbb{F}_{q_1}\). Hence
$$
B_0,B_1\in H_1
$$
because their determinants belong to $D$, while
$$
B_2,\ldots,B_d\in\mathrm{SL}_d(q_1)\leq H_1.
$$

Therefore $
B_0,B_1,\ldots,B_d\in G_0.$ Similar calculations to the proofs of Lemmas~\ref{l:Bidef} and~\ref{l:B0def} show that $I\sim_{d+1}J$ whereas clearly $I\not\sim_{d+2}J$ since $v\ne w$. Hence $\rc(V\rtimes G_0)\ge d+2$. The reverse inequality follows
from Theorem~\ref{thrm:main}, and therefore
$\rc(V\rtimes G_0)=d+2$.
\end{proof}

It remains to consider the groups as in part~\eqref{part3}.

\begin{lemma}
    Suppose that $d= 2$, $p=3$, $f$ is even and that
$G_0=NZ$, where $N=\mathrm{SL}_2(5)$ is a maximal subgroup of $\mathrm{SL}_2(9)$ and $Z\le\mathbb{F}_q^\ast$ is a group of scalar matrices containing an element of order $4$. Then $\rc(G)=d+2=4.$
\end{lemma}
\begin{proof}
Since $\mathbb{F}_9\le\mathbb{F}_q$, we may choose $\iota\in\mathbb{F}_9$ such that
$\iota^2=-1$
and $\iota I_2\in G_0$.
Let $
\overline{N}=N/\{\pm I_2\}\cong \mathrm{Alt}(5)$.
Choose an element $\overline{t}\in\overline{N}$ of order $3$. The normalizer of $\langle\overline{t}\rangle$ in $\overline{N}$
is isomorphic to $\mathrm{Sym}(3)$. So there exists an
involution $\overline{x}\in\overline{N}$ such that
$\overline{x}\,\overline{t}\,\overline{x}^{-1}
=
\overline{t}^{-1}$.
Choose lifts $t,x\in N$ such that $t$ has order $3$. Then $t$ is a
transvection, while
$x^2=-I_2$ and $x^{-1}tx=t^{-1}$.

After choosing a suitable basis $e_1,e_2$ of $V$, we may assume that
$$
t=
\begin{pmatrix}
1&0\\
1&1
\end{pmatrix}
\hbox{ and }
x=
\iota
\begin{pmatrix}
-1&0\\
0&1
\end{pmatrix}.
$$
Define
$$
B_1=\iota^{-1}x=
\begin{pmatrix}
-1&0\\
0&1
\end{pmatrix},
\qquad
B_2=t=
\begin{pmatrix}
1&0\\
1&1
\end{pmatrix},
$$
and
$$
B_0=B_1B_2=
\begin{pmatrix}
-1&0\\
1&1
\end{pmatrix}.
$$
Since $\iota I_2\in G_0$ and $x,t\in N$, all three matrices belong to
$G_0$.

Now set $
v=e_1+e_2$ and $w=-e_1+e_2$.
Then
$v\ne w$,
and a direct calculation gives
$$
(0,e_2,v)^{B_1}=(0,e_2,w)
$$
and
$$
(0,e_1,v)^{B_2}=(0,e_1,w).
$$
Moreover, the affine transformation
$
(-e_1,B_0)
$
satisfies
$
(e_1,e_2,v)^{(-e_1,B_0)}
=
(e_1,e_2,w).
$
Thus $(0,e_1,e_2,v)\sim_3(0,e_1,e_2,w)$ and $(0,e_1,e_2,v)\not\sim_4(0,e_1,e_2,w)$. Hence $\rc(G)\ge4$. The reverse inequality follows from
Theorem~\ref{thrm:main}, and therefore $\rc(G)=4$. 
\end{proof}
This completes the proof of Theorem~\ref{thrm:main2}.
\section{Some affine groups with $q=2$ having relational complexity $d+1$}\label{sec:someq=2}
In this section we discuss the relational complexity of some primitive subgroups $G=V\rtimes G_0$ of $\mathrm{AGL}_d(2)$, which are relevant for the proof of Theorem~\ref{thrm:main3}. Throughout, we assume $d\ge 2$. The goal of the section is to prove the following:

\begin{proposition}\label{thrm:calculations}
    Let $G=V\rtimes G_0$ where $V$ is a vector space of dimension $d\geq 2$ over $\mathbb{F}_2$ and $G_0$ is an irreducible subgroup of $\mathrm{SL}_d(2)$. If $G_0$ is one of
    \begin{enumerate}
        \item $\mathrm{SL}_d(2)$, $d$ odd;
        \item $\Omega^\pm_d(2)$, $d\geq 6$ even; or
        \item $\mathrm{Alt}(m)$, $m\in\{d+1,d+2\}$, $d$ even, acting on the fully deleted permutation module,
    \end{enumerate}
    then $\mathrm{RC}(G)=d+1$.
    If $G_0$ is one of
    
    \begin{enumerate}
        \item[(4)] $\mathrm{Sp}_d(2)$, $d\geq 4$;
        \item[(5)] $\mathrm{O}^\pm_d(2)$, $d\geq 4$; or
        \item[(6)] $\mathrm{Sym}(m)$, $m\in\{d+1,d+2\}$, $d$ even, acting on the fully deleted permutation module,
    \end{enumerate}
    then $\mathrm{RC}(G)\ne d+1$.
\end{proposition}
Each case shall be treated separately, although we will use the same setup for (4), (5), and (6).
\subsection{Setup for Cases (4), (5), and (6) of Proposition~\ref{thrm:calculations}}\label{sec:setup} Assume that $G_0$ is as in (4), (5), or (6) and suppose for a contradiction that $\rc(G)=d+1$. Then there exist
$I,J\in V^{d+1}$ such that \(I\sim_d J\) and $I\not\sim_{d+1}J$.
By Lemma~\ref{l: almost independent}, we may suppose that the supports of $I$ and $J$
are almost independent sets of cardinality $d+1$.
In each case $d$ is even, so up to replacing $G$, $I$ and $J$ by suitable $\mathrm{AGL}_d(2)$-conjugates, from Lemma~\ref{l:ai}, we may suppose that 
$I=(0,e_1,\ldots,e_{d-1},e_d)$ and $J=(0,e_1,\ldots,e_{d-1},w)$, where $e_1,\ldots,e_{d-1},w$ form a basis of $V$. Let $\delta=w+e_d\in \langle e_1,\ldots,e_{d-1}\rangle$.
For each $i \in \{0, \ldots, d\}$, choose an affine transformation $g_i = (t_i,B_i ) \in G$ witnessing the $d$-equivalence of $I$ and $J$ after deleting the $(i+1)$-th coordinate. In particular, $t_i=0$ for $i\ne 0$. Write $U=\langle e_1,\dots, e_{d-1}\rangle$.
\begin{lemma}\label{l:formpres}
    Suppose that $d\geq 4$. Let $\langle\, , \rangle$ denote the natural nondegenerate alternating or polar form on $V$ preserved by $G_0$. Then $U=\delta^\perp$, where orthogonality is with respect to the form $\langle\, , \rangle$. Therefore, $\langle e_d,\delta\rangle=1$.
\end{lemma}
\begin{proof}
Fix $j\in\{1,\ldots,d-1\}$. Since $d\geq4$, we may choose $i\in\{1,\ldots,d-1\}\setminus\{j\}$.
The matrix $B_i$ fixes $e_j$ and maps
$e_d$ to $w=e_d+\delta$.
Since $B_i$ preserves the form,
\[
\langle e_j,e_d\rangle
=
\langle e_jB_i,e_dB_i\rangle
=
\langle e_j,e_d+\delta\rangle.
\]
It follows that $\langle e_j,\delta\rangle=0$.
This holds for every $j\in\{1,\ldots,d-1\}$, and hence $U\leq\delta^\perp$.
Since $\delta\neq0$ and the form is nondegenerate,
$\dim\delta^\perp=d-1$.
As $\dim U=d-1$, we obtain $U=\delta^\perp$, as claimed.
In particular, $e_d\notin\delta^\perp$, and hence $\langle e_d,\delta\rangle=1$.
\end{proof}
\subsection{The affine general linear group}\label{sec:5.1}
When $G=V\rtimes\mathrm{SL}_d(2)$, we have seen in Theorem~\ref{thrm:main} that $\rc(G)=d+1$ if and only if $d$ is odd and hence part~\eqref{part12} of Proposition~\ref{thrm:calculations} immediately follows.

\subsection{The symplectic groups $G_0=\mathrm{Sp}_d(2)$}\label{sec:5.2}
Suppose, for a contradiction, that $\mathrm{RC}(G)=d+1$, and assume the setup and notation from Section~\ref{sec:setup}. When $d=2$, we have $\mathrm{Sp}_2(2)=\mathrm{SL}_2(2)$,
so we may assume that $d\geq4$. Let $\langle\,,\rangle$ denote the nondegenerate alternating form preserved by $G_0$. By Lemma~\ref{l:formpres}, $U=\delta^\perp$ and $\langle e_d,\delta \rangle =1.$

Now consider the symplectic transvection
$\tau_\delta:x\longmapsto
x+\langle x,\delta\rangle\delta$.
We have $\tau_\delta\in G_0$.
Since $U=\delta^\perp$, the transformation $\tau_\delta$ fixes
$e_1,\ldots,e_{d-1}$, while
$$
e_d\tau_\delta
=
e_d+\langle e_d,\delta\rangle\delta
=
e_d+\delta
=
w.
$$
Thus $\tau_\delta$ maps the entire tuple $I$ to the entire tuple $J$,
contrary to the assumption that $I$ and $J$ are not
$G$-conjugate.
Therefore, $G_0\neq\mathrm{Sp}_d(2)$.

\subsection{The orthogonal groups $G_0=\mathrm{O}_d^\pm(2)$}\label{sec:5}
Suppose, for a contradiction, that $\mathrm{RC}(G)=d+1$. Let $Q:V\longrightarrow\mathbb F_2$
be the nondegenerate quadratic form preserved by $G_0$, and let
$p(x,y)=Q(x+y)+Q(x)+Q(y)$ be its associated polar form. Assume the setup and notation from Section~\ref{sec:setup}. In dimension $2$, the group $\mathrm{O}_2^-(2)$ coincides with
$\mathrm{SL}_2(2)$, while $\mathrm{O}_2^+(2)$ is reducible.  We
may therefore assume that $d\geq4$. By Lemma~\ref{l:formpres}, $U=\delta^\perp$ and $p( e_d,\delta ) =1.$

Now $B_i$ preserves $Q$ and maps $e_d$ to $e_d+\delta$. Therefore,
$Q(e_d)=Q(e_d+\delta)$.
Using the definition of the polar form, we obtain
$$
Q(e_d+\delta)
=
Q(e_d)+Q(\delta)+p(e_d,\delta).
$$
Thus $Q(\delta)+p(e_d,\delta)=0$.
Since $p(e_d,\delta)=1$, we deduce that $Q(\delta)=1$.
It follows that the orthogonal transvection $
\tau_\delta:x\longmapsto x+p(x,\delta)\delta$
belongs to $G_0$.
Since  $U=\delta^\perp$,
the element $\tau_\delta$ fixes $e_1,\ldots,e_{d-1}$.
Moreover,
$$
e_d\tau_\delta
=
e_d+p(e_d,\delta)\delta
=
e_d+\delta
=
w.
$$
Hence $\tau_\delta$ maps the whole tuple $I$ to the whole tuple $J$,
contrary to the assumption that $I$ and $J$ are not $G$-conjugate.
Therefore, $G_0\neq\mathrm{O}_d^\pm(2)$.

\subsection{The ``omega'' groups $G_0=\Omega_d^\pm(2)$}\label{sec:omega}
\begin{lemma}\label{l:non-singular-basis}
Let $(V,Q)$ be a nondegenerate orthogonal space of even dimension $d$
over $\mathbb F_2$. Suppose either that $V$ has minus type, or that
$V$ has plus type and $d\geq6$. Then $V$ has a basis
$f_1,\ldots,f_d$ such that
$Q(f_i)=1$, for every $i\in\{1,\ldots,d\}$ and
$Q(f_1+\cdots+f_d)=1$.
\end{lemma}

\begin{proof}
Let $p$ denote the associated polar form. We first describe a four-dimensional plus-type space which will be
used to increase the dimension by $4$. Let
$P=\langle \alpha,\beta\rangle\perp\langle \gamma,\delta\rangle$,
where $(\alpha,\beta)$ and $(\gamma,\delta)$ are hyperbolic pairs. Thus
\[
Q(\alpha)=Q(\beta)=Q(\gamma)=Q(\delta)=0,
\qquad
p(\alpha,\beta)=p(\gamma,\delta)=1,
\]
and all the other pairings among these vectors are zero. Set
\[
z_1=\gamma+\delta,\qquad
z_2=\beta+\gamma+\delta,\qquad
z_3=\alpha+\gamma+\delta,\qquad
z_4=\alpha+\beta+\delta.
\]
Then $z_1,z_2,z_3,z_4$ form a basis of $P$ and
\[
Q(z_i)=1
\]
for every $i$. Moreover,
\[
z_1+z_2+z_3+z_4=\gamma,
\]
and hence
\[
Q(z_1+z_2+z_3+z_4)=0.
\]

Consequently, suppose that a nondegenerate orthogonal space $W$ has a
basis $x_1,\ldots,x_n$ satisfying
\[
Q(x_i)=1
\quad\text{and}\quad
Q(x_1+\cdots+x_n)=1.
\]
Then
\[
x_1,\ldots,x_n,z_1,z_2,z_3,z_4
\]
is a basis of $W\perp P$ consisting of non-singular vectors, and
\[
Q\left(
x_1+\cdots+x_n+z_1+\cdots+z_4
\right)
=
Q(x_1+\cdots+x_n)+Q(\gamma)=1.
\]
Since $P$ has plus type, $W\perp P$ has the same type as $W$.

It therefore suffices to give the required bases in dimensions $2$
and $4$ for the minus type, and in dimensions $6$ and $8$ for the
plus type.

For the two-dimensional minus-type space, let $\alpha,\beta$ be an anisotropic
basis, so that
\[
Q(\alpha)=Q(\beta)=Q(\alpha+\beta)=1.
\]
Then $\alpha,\beta$ is the required basis.

For the four-dimensional minus-type space, write
\[
V=\langle \alpha,\beta\rangle\perp\langle \gamma,\delta\rangle,
\]
where $\langle \alpha,\beta\rangle$ is anisotropic and $(\gamma,\delta)$ is a hyperbolic
pair. Then
\[
\gamma+\delta,\qquad \beta,\qquad \beta+\delta,\qquad \alpha
\]
is a basis consisting of non-singular vectors, and its sum is $\alpha+\gamma$,
which is non-singular.

For the six-dimensional plus-type space, write
\[
V=
\langle \alpha,\beta\rangle
\perp
\langle \gamma,\delta\rangle
\perp
\langle \varepsilon,\zeta\rangle,
\]
where the three displayed pairs are hyperbolic. The vectors
\[
\begin{split}
&\varepsilon+\zeta,\qquad \delta+\varepsilon+\zeta,\qquad \gamma+\varepsilon+\zeta,\\
&\gamma+\delta+\zeta,\qquad \beta+\varepsilon+\zeta,\qquad \alpha+\varepsilon+\zeta
\end{split}
\]
form a basis of $V$ and are all non-singular. Their sum is
\[
\alpha+\beta+\varepsilon,
\]
which is non-singular.

Finally, for the eight-dimensional plus-type space, write
\[
V=
\langle \alpha,\beta\rangle
\perp
\langle \gamma,\delta\rangle
\perp
\langle \varepsilon,\zeta\rangle
\perp
\langle \eta,\mu\rangle,
\]
where all four displayed pairs are hyperbolic. The vectors
\[
\begin{split}
&\eta+\mu,\qquad \zeta+\eta+\mu,\qquad \varepsilon+\eta+\mu,\qquad \varepsilon+\zeta+\mu,\\
&\delta+\eta+\mu,\qquad \gamma+\eta+\mu,\qquad \beta+\eta+\mu,\qquad \alpha+\varepsilon+\zeta
\end{split}
\]
form a basis and are all non-singular. Their sum is
\[
\alpha+\beta+\gamma+\delta+\varepsilon+\zeta+\mu,
\]
which is non-singular.

The result now follows inductively by taking an orthogonal direct sum
with the four-dimensional space $P$. Indeed, the base dimensions
$2$ and $4$ cover all even dimensions of minus type, while the base
dimensions $6$ and $8$ cover all even dimensions at least $6$ of
plus type.
\end{proof}

\begin{proposition}\label{prop:converse-omega}
Let $V$ be a nondegenerate orthogonal space of even dimension $d$ over
$\mathbb F_2$, and let
$G=V\rtimes\Omega_d^\varepsilon(2)$.
Suppose that either $\varepsilon=-$, or that $\varepsilon=+$ and
$d\geq6$. Then $\rc(G)=d+1$.
\end{proposition}

\begin{proof}
Let $Q$ be the quadratic form on $V$, and let $\langle x,y\rangle
=
Q(x+y)+Q(x)+Q(y)$
be its associated polar form. By Lemma~\ref{l:non-singular-basis}, there is a basis
$f_1,\ldots,f_d$ of $V$ such that $Q(f_i)=1$
for every $i$ and $
Q(f_0)=1$, where $f_0=f_1+\cdots+f_d$.

Let $e_1,\ldots,e_d$ be the basis dual to
$f_1,\ldots,f_d$ with respect to the polar form. Thus
$\langle e_i,f_j\rangle=\delta_{ij}$.
In particular, 
$\langle e_j,f_0\rangle=1$,
for every $j\in\{1,\ldots,d\}$.

For every non-singular vector $f\in V$, define
$\rho_f:x\longmapsto x+\langle x,f\rangle f$.
Since $Q(f)=1$, the map $\rho_f$ preserves $Q$. Moreover,
$\rho_f-I_d$ has rank $1$, and hence $\rho_f$ has non-trivial Dickson
invariant. Therefore
$\rho_f\in
\mathrm O_d^\varepsilon(2)\setminus
\Omega_d^\varepsilon(2)$.

Set $\widehat G=V\rtimes\mathrm O_d^\varepsilon(2)$.
Then $G\unlhd\widehat G$ and $[\widehat G:G]=2$.

Consider the $(d+1)$-tuple
\[
I=(0,e_1,\ldots,e_d).
\]
Set $s=(0,\rho_{f_1})\in\widehat G\setminus G$
and define $J=I^s$.

For $i\in\{1,\ldots,d\}$, set
\[
r_i=(0,\rho_{f_i}).
\]
The element $r_i$ fixes $0$ and fixes every $e_j$ with $j\neq i$,
because $\langle e_j,f_i\rangle=0$,
when $j\neq i$. Thus $r_i$ fixes pointwise the tuple obtained from
$I$ by deleting the coordinate $e_i$.

For the coordinate containing $0$, set
$r_0=(f_0,\rho_{f_0})$.
For every $j\in\{1,\ldots,d\}$, we have
\[
e_j^{r_0}
=
e_j+\langle e_j,f_0\rangle f_0+f_0
=
e_j+f_0+f_0
=
e_j.
\]
Thus $r_0$ fixes pointwise the tuple obtained from $I$ by deleting its
first coordinate.

Every $r_i$ belongs to $\widehat G\setminus G$. Therefore
$g_i=r_is\in G$,
for every $i\in\{0,\ldots,d\}$. If $I_{\widehat i}$ and
$J_{\widehat i}$ denote the tuples obtained by deleting the $i$-th
coordinate, then
$$
I_{\widehat i}^{\,g_i}
=
I_{\widehat i}^{\,r_is}
=
I_{\widehat i}^{\,s}
=
J_{\widehat i}.
$$
Therefore, $
I\sim_d J.$

We claim that $I$ and $J$ are not $G$-conjugate. Suppose that
$(v,A)\in G$ maps $I$ to $J$. Since the first coordinate of both
tuples is $0$, we have $v=0$. Moreover, $e_jA=e_j\rho_{f_1}$,
for every $j\in\{1,\ldots,d\}$. Since $e_1,\ldots,e_d$ is a basis,
this implies  $A=\rho_{f_1}$.
However,  $\rho_{f_1}\notin\Omega_d^\varepsilon(2)$,
a contradiction. Hence $ I\not\sim_{d+1}J$.

It follows that $\rc(G)\geq d+1$. The general upper bound
$\rc(G)\leq d+1$ gives $\rc(G)=d+1$.
\end{proof}

\subsection{Symmetric groups in their action on the fully deleted permutation modules}\label{sec:5.5}
Suppose, for a contradiction, that $\mathrm{RC}(G)=d+1$, and assume the setup and notation from Section~\ref{sec:setup}. When $m\in \{3,4\}$, the natural image of $\mathrm{Sym}(m)$ is $\mathrm{SL}_2(2)$, so these cases are covered earlier. Moreover,
the low-dimensional coincidences $\mathrm{Sym}(5)\cong\mathrm{O}_4^-(2)$, $\mathrm{Sym}(6)\cong\mathrm{Sp}_4(2)$ and $\mathrm{Sym}(8)\cong\mathrm{O}_6^+(2)$
are covered by the orthogonal and symplectic cases. We may therefore
assume that $m\geq 7$ and, when $m$ is even, that $m\geq10$. The fully deleted permutation module has a natural nondegenerate
$G_0$-invariant alternating form, which we denote by $\langle\ ,\ \rangle$. By Lemma~\ref{l:formpres}, $U=\delta^\perp$ and $\langle e_d,\delta \rangle =1.$

Let $\mathcal R$ be the set of directions of the transvections in $G_0$.
In the fully deleted permutation realization,
\[
\mathcal R=
\left\{
\overline{f_\alpha+f_\beta}\mid 1\leq \alpha<\beta\leq m
\right\},
\]
where $f_1,\ldots,f_m$ is the standard basis of the permutation module
and the bar denotes the corresponding vector in the fully deleted module.
We claim that $\delta\notin\mathcal R$.
Indeed, if $\delta\in\mathcal R$, then the transposition corresponding
to $\delta$ induces the transvection
$\tau_\delta:x\longmapsto
x+\langle x,\delta\rangle\delta$.
Since $U=\delta^\perp$, the transformation $\tau_\delta$ fixes
$e_1,\ldots,e_{d-1}$ and maps $e_d$ to $e_d+\delta=w$.
Thus $\tau_\delta$ maps the whole tuple $I$ to the whole tuple $J$, a
contradiction. Hence $\delta$ is not a transvection direction.

For every $i\in\{0,\ldots,d-1\}$, we have $
\delta\in\operatorname{Im}(B_i-I_d).$
Indeed, for $i\geq1$, this follows from $e_d(B_i-I_d)=\delta$.
For $i=0$, we have $(e_d+e_1)(B_0-I_d)=\delta$.
Moreover, $1\leq\mathrm{rank}(B_i-I_d)\leq2$. Similarly to before we say that $\mathrm{Im}(B_i-I_d)$ is the \textit{\textbf{moved space}} of $B_i$ for $i=0,\dots, d-1$.

We now use the elements of $\mathrm{Sym}(m)$ having rank at most $2$
on the fully deleted permutation module. Under our assumptions on $m$, such
a nonidentity element is one of the following:
\begin{enumerate}
\item a transposition, of rank $1$;
\item a $3$-cycle, of rank $2$;
\item a product of two disjoint transpositions, of rank $2$.
\end{enumerate}

If there exists $i\in\{0,\dots, d-1\}$ such that $B_i$ is a transposition, then (by the definition of $\mathcal R$) the moved space of $B_i$ is a line spanned by a vector in to $\mathcal R$. On the other hand we have seen that this line would also need to equal $\langle\delta\rangle$ and we have seen that $\delta \not\in \mathcal R$. Thus none of the $B_i$ are transvections.

We claim that none of the $B_i$ are $3$-cycles. Indeed, if the $3$-cycle is
$(\alpha\,\beta\,\gamma)$, then its moved space has non-zero vectors
$\overline{f_\alpha+f_\beta},\overline{f_\beta+f_\gamma},\overline{f_\alpha+f_\gamma}$,
all of which belong to $\mathcal R$. Since $\delta$ belongs to this
moved space, this would imply $\delta\in\mathcal R$, a contradiction.

Therefore, every $B_i$ is the image of a product of two disjoint
transpositions for $i\in\{0,\dots, d-1\}$. If $B_i=(\alpha\,\beta)(\gamma\,\zeta)$,
then
$$
\mathrm{Im}(B_i-I_d)
=
\left\langle
\overline{f_\alpha+f_\beta},\,
\overline{f_\gamma+f_\zeta}
\right\rangle.
$$
Its two displayed non-zero vectors belong to $\mathcal R$, while its
third non-zero vector is $\overline{f_\alpha+f_\beta+f_\gamma+f_\zeta}$.
Since $\delta\notin\mathcal R$, we must have
$\delta=
\overline{f_\alpha+f_\beta+f_\gamma+f_\zeta}$.
Thus $\delta$ determines the four-point support
$\{\alpha,\beta,\gamma,\zeta\}$.
There are exactly three products of two disjoint transpositions on
this support:
$$
(\alpha\,\beta)(\gamma\,\zeta),\qquad
(\alpha\,\gamma)(\beta\,\zeta),\qquad
(\alpha\,\zeta)(\beta\,\gamma).
$$
Therefore, there are at most three possible matrices $B_i$.

On the other hand, the matrices $B_1,\ldots,B_{d-1}$
are pairwise distinct. Indeed, if $B_i=B_j$ for distinct
$i,j\in\{1,\ldots,d-1\}$, then $B_i=B_j$ maps $I$ to $J$, contrary to $I\not\sim_{d+1}J$.
Hence there are at least $d-1$ distinct possible matrices $B_i$.  Since $m\ge 7$ and $d$ is even, we have $d\ge 6$, and hence $d-1>3$, a contradiction.
Therefore neither $G_0\cong\mathrm{Sym}(d+1)$
nor $G_0\cong\mathrm{Sym}(d+2)$
can occur.

\subsection{The alternating groups in their action on the fully deleted permutation module}\label{sec:alt}
\begin{proposition}\label{prop:converse-alternating}
Let $d\geq 2$ be even, let $m\in\{d+1,d+2\}$, and let $G_0$ be the
natural image of $\mathrm{Alt}(m)$ on the fully deleted permutation module
$V$ over $\mathbb F_2$. Set 
$G=V\rtimes G_0$.
Then  $\rc(G)=d+1$.
\end{proposition}

\begin{proof}
Let $E=\mathbb F_2^m$, with standard basis
$f_1,\ldots,f_m$, and set
\[
E_0=
\left\{
\sum_{i=1}^m x_if_i\in E
\mathrel{}\middle|\mathrel{}
\sum_{i=1}^m x_i=0
\right\}.
\]
When $m=d+1$, we identify $V$ with $E_0$. When
$m=d+2$, we identify $V$ with $E_0/\langle f_1+\cdots+f_m\rangle$.
In both cases, we denote by $\overline{x}$ the vector of $V$
corresponding to $x\in E_0$.

The usual scalar product on $E$ induces a nondegenerate alternating
bilinear form on $V$, which we denote by $\langle\ ,\ \rangle$.
For distinct $a,b\in\{1,\ldots,m\}$, set $
z_{a,b}=\overline{f_a+f_b}$.
The transposition $(a\,b)$ induces on $V$ the transvection 
$\tau_{a,b}:x\longmapsto
x+\langle x,z_{a,b}\rangle z_{a,b}$.

For every $i\in\{1,\ldots,d\}$, set $z_i=z_{i,i+1}
=
\overline{f_i+f_{i+1}}$.
The vectors $z_1,\ldots,z_d$ form a basis of $V$.

Set $z_0=z_1+\cdots+z_d$.
The intermediate terms cancel, and therefore $z_0=\overline{f_1+f_{d+1}}$.
In particular, $z_0$ is also the direction of a transvection induced
by a transposition.

Let $e_1,\ldots,e_d$ be the basis of $V$ dual to
$z_1,\ldots,z_d$ with respect to $\langle\ ,\ \rangle$. Thus
$\langle e_i,z_j\rangle=\delta_{ij}$.
Since $z_0=z_1+\cdots+z_d$, we have
\begin{equation}\label{eq:dual-z0}
\langle e_j,z_0\rangle=1
\end{equation}
for every $j\in\{1,\ldots,d\}$.

Let $\widehat G_0$ be the natural image of $\mathrm{Sym}(m)$ on $V$,
and set $\widehat G=V\rtimes\widehat G_0$.
The kernel of the action of $\mathrm{Sym}(m)$ on $V$ is contained in
$\mathrm{Alt}(m)$. Indeed, the action is faithful except when $m=4$,
in which case its kernel is the Klein four group. Therefore, $G\unlhd\widehat G$  and $[\widehat G:G]=2$.
In particular, the image of every transposition belongs to
$\widehat G_0\setminus G_0$.

For $i\in\{1,\ldots,d\}$, let $\tau_i$ be the transvection induced by
the transposition $(i\,\,i+1)$, so that
$x\tau_i=x+\langle x,z_i\rangle z_i$.
Let $\tau_0$ be the transvection induced by the transposition
$(1\,\,d+1)$, so that $x\tau_0=x+\langle x,z_0\rangle z_0$.

Consider the $(d+1)$-tuple
\[
I=(0,e_1,\ldots,e_d).
\]
Set $s=(0,\tau_1)\in\widehat G\setminus G$
and define $J=I^s$.

For $i\in\{1,\ldots,d\}$, set $r_i=(0,\tau_i)$.
If $j\neq i$, then
\[
e_j^{r_i}
=
e_j+\langle e_j,z_i\rangle z_i
=
e_j.
\]
Moreover, $r_i$ fixes $0$. Thus $r_i$ fixes pointwise the tuple
obtained from $I$ by deleting the entry $e_i$.

For the coordinate containing $0$, set $r_0=(z_0,\tau_0)$.
By~\eqref{eq:dual-z0}, for every $j\in\{1,\ldots,d\}$ we have
\[
e_j^{r_0}
=
e_j+\langle e_j,z_0\rangle z_0+z_0
=
e_j+z_0+z_0
=
e_j.
\]
Thus $r_0$ fixes pointwise the tuple obtained from $I$ by deleting its
first coordinate.

Every $r_i$ and $s$ belongs to $\widehat G\setminus G$. Since
$G$ has index $2$ in $\widehat G$, it follows that
$g_i=r_is\in G$,
for every $i\in\{0,\ldots,d\}$.

Let $I_{\widehat i}$ and $J_{\widehat i}$ denote the $d$-tuples
obtained from $I$ and $J$, respectively, by deleting the $i$-th
coordinate. Since $r_i$ fixes $I_{\widehat i}$ pointwise, we have
\[
I_{\widehat i}^{\,g_i}
=
I_{\widehat i}^{\,r_is}
=
I_{\widehat i}^{\,s}
=
J_{\widehat i}.
\]
Therefore
\[
I\sim_d J.
\]

We finally prove that $I$ and $J$ are not $G$-conjugate. Suppose that
$(v,A)\in G$ maps $I$ to $J$. Since the first coordinate of both
tuples is $0$, we have $v=0$.
Moreover, $e_jA=e_j\tau_1$,
for every $j\in\{1,\ldots,d\}$. Since
$e_1,\ldots,e_d$ form a basis of $V$, this gives  $A=\tau_1$.
However, $\tau_1$ is induced by a transposition and hence $\tau_1\notin G_0$,
a contradiction. Thus $I\not\sim_{d+1}J$.

It follows that $\rc(G)\geq d+1$.
The general upper bound $\rc(G)\leq d+1$ now gives $\rc(G)=d+1$.
\end{proof}

The proof of Proposition~\ref{prop:converse-alternating} completes the proof of Proposition~\ref{thrm:calculations}.

\section{Affine groups with $q=2$ having maximum relational complexity}\label{sec:maximalq=2}
In this section, we prove Theorem~\ref{thrm:main3}.
We have verified Theorem~\ref{thrm:main3} with a computer when $d\le 8$. In particular, along with Proposition~\ref{thrm:calculations}, this completes the proof of the reverse implication in Theorem~\ref{thrm:main3}: we confirm that if $G_0$ is as in parts (1), (2), (3) or (4), then $G=V\rtimes G_0$ has relational complexity $d+1$.

Thus assume that $d>8$, and let
$G=V\rtimes G_0\leq \mathrm{AGL}_d(2)$, 
where $G_0\leq \mathrm{GL}_d(2)$ acts irreducibly on $V$. Suppose that $\rc(G)=d+1$.
We aim to classify all such groups $G$, and thereby prove
Theorem~\ref{thrm:main3}.

By Lemma~\ref{l: almost independent}, there exist two  $(d+1)$-tuples $I$ and $J$
such that $I\sim_d J$ and  $I\not\sim_{d+1}J$,
and such that the underlying sets of $I$ and $J$ are almost independent sets of
cardinality $d+1$. Write
$
I=(P_0,\ldots,P_{d-1},P_d)$ and 
$J=(Q_0,\ldots,Q_{d-1},Q_d).$
Since $I\sim_d J$, there exists an element of $G$ mapping $(P_0,\ldots,P_{d-1})$ to $
(Q_0,\ldots,Q_{d-1})$.
Thus, after replacing $J$ by a suitable $G$-conjugate, we may assume that $P_i=Q_i$, for every $i\in\{0,\ldots,d-1\}$.
The set $\{P_0,\ldots,P_{d-1}\}$
is independent for $G$ since the underlying set of $I$ is an almost independent set. Moreover, an independent set for $G$ is also
independent for $\mathrm{AGL}_d(2)$. Indeed, an element of $G$ witnessing
the strict containment of two pointwise stabilizers also witnesses the
corresponding strict containment in $\mathrm{AGL}_d(2)$. Consequently, by
Lemma~\ref{l:0}, the set $\{P_0,\ldots,P_{d-1}\}$ is a frame for an
affine subspace of dimension $d-1$.

After conjugating $G$, $I$, and $J$ simultaneously by a suitable element of
$\mathrm{AGL}_d(2)$, we may therefore assume that 
$I=\left(0,e_1,\ldots,e_{d-1},v\right)$ and $J=\left(0,e_1,\ldots,e_{d-1},w\right)$,
for some $v,w\in V$.

Lemma~\ref{l:ai} therefore implies that precisely one
of the following holds for $v$:
\begin{itemize}
    \item the vectors
    $e_1,\ldots,e_{d-1},v$
    form a basis of $V$;
    \item $d$ is odd and
    $v=\sum_{i=1}^{d-1}e_i$.
\end{itemize}
The analogous alternatives hold for $w$.

The second alternative cannot hold simultaneously for both $v$ and $w$, since
in that case $v=w$, and hence $I=J$, contradicting
$I\not\sim_{d+1}J$.
Consequently, after interchanging $I$ and $J$ if necessary, we may assume that
$e_1,\ldots,e_{d-1},v$
is a basis of $V$. In particular, we may suppose $v=e_d$. We are therefore in one of the following two cases:
\begin{enumerate}
    \item\label{casecase1}$
    e_1,\ldots,e_{d-1},w$
    forms a  basis of $V$;
    \item $d$ is odd and
    $w=\sum_{i=1}^{d-1}e_i$.
\end{enumerate}

We consider these two cases separately, beginning with the case $w=\sum_{i=1}^{d-1}e_i$. Before proceeding, we start with a technical lemma, to be used throughout.

\begin{lemma}\label{l:Vdecomp}
    Let $1\ne T\trianglelefteq H$ where $H$ is irreducible on the module $V$ over $\mathbb{F}_2$. If $T$ is generated by transvections, then there is a decomposition $$V=V_1\oplus V_2\oplus\cdots\oplus V_m$$ where the $V_i$ are irreducible $T$-modules of equal dimension $n>1$ and $H$ permutes them transitively.
\end{lemma}
\begin{proof}
If $V$ is irreducible as a $T$-module then the result is trivial, so suppose not. Since $T\unlhd H$, $V$ is completely reducible as a $T$-module.
Let $U_1,\ldots,U_m$ be the homogeneous components of $V$ as a
$T$-module. The group $H$ permutes these components. Moreover,
this action is transitive.
We claim that every homogeneous component has multiplicity one. Fix
one such component $U$, and write $U\cong W^{\oplus r}$,
for some irreducible $T$-module $W$ and some positive integer
$r$. Since $T\neq 1$, we have ${\bf C}_V(T)\neq V$. On the other hand,
${\bf C}_V(T)$ is $H$-invariant, and hence the irreducibility of $H$ gives ${\bf C}_V(T)=0$.
In particular, $T$ acts non-trivially on $U$. Since $T$ is
generated by transvections, there exists a transvection $t\in T$
which acts non-trivially on $U$. Under the decomposition $U\cong W^{\oplus r}$, the action of $t$ on $U$ is the direct
sum of $r$ copies of its action on $W$. Therefore $\mathrm{rank}((t-I_d)|_U)
 =r\cdot\mathrm{rank}((t-I_d)|_W) \geq r$.
Since $t$ is a transvection on $V$, we have $\mathrm{rank}(t-I_d)=1$.
Therefore, $r=1$.
Thus every homogeneous component is irreducible as a $T$-module. We may therefore
write
\begin{equation}\label{eq:T-decomposition}
V=V_1\oplus\cdots\oplus V_m,
\end{equation}
where the $V_i$ are irreducible $T$-modules and $H$ permutes
them transitively. 

Set
$n=\dim_{\mathbb F_2}(V_i)$. Then $n>1$, as
$\mathrm{GL}_1(2)$ contains no transvections.\end{proof}

\subsection{Case $v=e_d$ and $w=\sum_{i=1}^{d-1}e_i$.}Set
$$
\mathbf{1}=w-v=w+v=\sum_{j=1}^{d}e_j,
$$
where $\mathbf{1}$ is the all one row vector and where the second equality follows from the fact that the underlying field is
$\mathbb F_2$.

For each $i\in\{0,\ldots,d\}$, choose an affine transformation
$g_i=(t_i,B_i)\in G$ witnessing the $d$-equivalence of $I$ and $J$
after deleting the $i$-th coordinate. Thus, $g_i$ maps the $d$-tuple
obtained from $I$ by deleting its $i$-th coordinate to the corresponding
$d$-tuple obtained from $J$.

After deleting the final coordinate, the resulting tuples are identical.
Consequently, we may take $B_d=I_d$ and $t_d=0$.

Now let $i\in\{1,\ldots,d-1\}$. After deleting the coordinate
corresponding to $e_i$, the affine transformation $g_i$ fixes $0$, and
hence $t_i=0$. Moreover,
\begin{equation}\label{eq:Bieven}
e_jB_i=e_j
\end{equation}
for every $j\in\{1,\ldots,d-1\}\setminus\{i\}$, while
\begin{equation}\label{eq:B0even}
e_dB_i=w=e_d+\mathbf{1}.
\end{equation}
Set
$$
a_i=e_i(B_i-I_d)=\sum_{j=1}^{d}a_{i,j}e_j.
$$
It follows from~\eqref{eq:Bieven} and~\eqref{eq:B0even} that
\begin{equation}\label{eq:binary-Bi}
B_i=I_d+e_i^\top a_i+e_d^\top \mathbf{1}.
\end{equation}

By the description of $B_i$, we have
$$
\begin{aligned}
\det(B_i)
&=
\det
\begin{pmatrix}
1+a_{i,i} & a_{i,d}\\
1 & 0
\end{pmatrix}
=
a_{i,d}.
\end{aligned}
$$
Since $B_i$ is invertible, $\det(B_i)=1$, and hence
\begin{equation}\label{eq:binary-aid}
a_{i,d}=1.
\end{equation}

We now consider the affine transformation $g_0=(t_0, B_0)$ obtained by
deleting the coordinate corresponding to $0$. It satisfies
\begin{equation}\label{eq:B0B0even}
e_jB_0+t_0=e_j
\end{equation}
for every $j\in\{1,\ldots,d-1\}$, and
\begin{equation}\label{eq:B0B0B0even}
e_dB_0+t_0=w.
\end{equation}
Set $t_0=a_0=\sum_{j}a_{0,j}e_j$. Since the characteristic is $2$,~\eqref{eq:B0B0even} and~\eqref{eq:B0B0B0even} yield
\begin{equation}\label{eq:binary-B0}
B_0=I_d+\mathbf{1}^\top a_0+e_d^\top \mathbf{1}.
\end{equation}
Since $d$ is odd, we have $ (\mathbf{1})(\mathbf{1}^\top )=1.$
Moreover, since $B_0$ is invertible and every entry of the final column is $a_{0,d}$, it follows that 
\begin{equation}\label{eq:binary-a0d}
a_{0,d}=1.
\end{equation}

For $i\in\{1,\ldots,d-1\}$, the column vectors $e_i^\top $ and $e_d^\top $
are linearly independent. Hence, from~\eqref{eq:binary-Bi},
\begin{equation}\label{eq:rank2}
\mathrm{rank}(B_i-I_d)
=
\dim_{\mathbb F_2}\langle a_i,\mathbf{1}\rangle.
\end{equation}
Similarly, from~\eqref{eq:binary-B0} since $\mathbf{1}^\top $ and $e_d^\top $ are linearly independent,
\begin{equation}\label{eq:rank2bis}
\mathrm{rank}(B_0-I_d)
=
\dim_{\mathbb F_2}\langle a_0,\mathbf{1}\rangle.
\end{equation}
Now,~\eqref{eq:binary-aid} and~\eqref{eq:binary-a0d} imply $a_i\ne0$, for all $i\in \{0,\ldots,d-1\}$. Therefore, for every $i\in\{0,\ldots,d-1\}$, it follows that
\begin{equation}\label{eq:rank-one-ai}
\mathrm{rank}(B_i-I_d)=1\hbox{ if and only if }
a_i=\mathbf{1}.
\end{equation}
In the rest of the proof we analyse the possible ranks of $B_i-I_d$. Throughout, we will repeatedly make use of the fact that if $B\in\mathrm{SL}_d(2)$ is such that $\mathrm{rank}(B-I_d)=1$, then $B$ is a transvection. Indeed, if $\mathrm{rank}(B-I_d)=1$ then $B=I_d+ba$ for some row vector $a\ne 0$ and column vector $b\ne 0$, thus $0\ne aB=a+aba=(1+ab)a$, and so $ab=0$.

\subsubsection{Suppose that
$\mathrm{rank}(B_i-I_d)=1$,
for every $i\in\{0,\ldots,d-1\}$.}\label{sec:subsub1}Here the argument is similar to the case $q>2$. By~\eqref{eq:rank-one-ai}, we have
$a_i=\mathbf{1}$ for every $i$. Therefore,
$B_0=I_d+(\mathbf{1}^\top +e_d^\top )\mathbf{1}$
and
$B_i=I_d+(e_i^\top +e_d^\top )\mathbf{1},$
for every $i\in\{1,\ldots,d-1\}$.

As $d$ is odd, we have $\mathbf{1}\cdot(\mathbf{1}^\top +e_d^\top )=0 $ and $\mathbf{1}\cdot(e_i^\top +e_d^\top )=0$, for every $i\in \{1,\ldots,d-1\}$.
Thus every $B_i$ is a transvection
with moved line $\langle\mathbf{1}\rangle$.

Let
$$
W=\langle \mathbf{1}^\top +e_d^\top ,e_1^\top +e_d^\top ,\ldots,e_{d-1}^\top +e_d^\top \rangle_{\mathbb F_2}.
$$
The vectors $e_1^\top +e_d^\top ,\ldots,e_{d-1}^\top +e_d^\top $ are linearly independent and belong to
$\ker(\mathbf{1})$, where we are identifying with $\mathbf{1}$ the linear mapping sending $z$ to $\mathbf{1}\cdot z^\top $. Since $\dim_{\mathbb F_2}\ker(\mathbf{1})=d-1$, it follows
that 
$W=\ker(\mathbf{1})$.
Therefore, $
E=\{I_d+c\mathbf{1}\mid c\in W\}$
is an elementary abelian subgroup of $G_0$ of order $2^{d-1}$, and every
nonidentity element of $E$ is a transvection. Moreover,
$
{\bf C}_V(E)=\langle\mathbf{1}\rangle.
$

Define
$N=\langle E^g\mid g\in G_0\rangle.$
Then $N\trianglelefteq G_0$, and $N$ is generated by transvections. Moreover, $N$ is irreducible on $V$ by Lemma~\ref{l:irredN}.

Thus $N$ is an irreducible subgroup of $\mathrm{GL}_d(2)$ generated by
transvections, and we may apply McLaughlin's classification of such groups~\cite{McL69}. In particular, up to conjugacy in
$\mathrm{GL}_d(2)$, one of the following holds:
\begin{enumerate}
    \item $N=\mathrm{SL}_d(2)$;
    \item $d$ is even and $N=\mathrm{Sp}_d(2)$;
    \item $d$ is even and $N=\mathrm{O}_d^{\pm}(2)$; or
    \item $d$ is even and $N$ is the natural image of
    $\mathrm{Sym}(d+1)$ or $\mathrm{Sym}(d+2)$ on its fully deleted
    permutation module.
\end{enumerate}
But $d$ is odd, whence $N=\mathrm{SL}_d(2)$.
Since $\mathbb F_2^\ast=\{1\}$, we have
$\mathrm{SL}_d(2)=\mathrm{GL}_d(2)$ and hence
$G=\mathrm{AGL}_d(2)$. This gives the example in~\eqref{part12} of Theorem~\ref{thrm:main3}.

\subsubsection{Suppose that
$\mathrm{rank}(B_r-I_d)=1,$
for some $r\in\{0,\ldots,d-1\}$, while 
$\mathrm{rank}(B_s-I_d)=2,$
for some $s\in\{0,\ldots,d-1\}$.}\label{sec:subsub2}
As in the previous subsection, we shall find an irreducible subgroup generated by transvections and apply McLaughlin's classification. We start with a technical lemma.

\begin{lemma}
Let
$T=\left\langle x\in G_0\mid \mathrm{rank}(x-I_d)=1\right\rangle$. Then $T\neq 1$ is an irreducible subgroup of $G_0$ generated by transvections.
\end{lemma}
\begin{proof}
Clearly, $T\trianglelefteq G_0$, and $B_r\in T$ so $T\ne 1$. Moreover, $T$ is generated by transvections, so there is a decomposition $V=V_1\oplus V_2\oplus\cdots\oplus V_m$ where the $V_i$ are irreducible $T$-modules of equal odd dimension $n\geq 3$ and $G_0$ permutes them transitively, by Lemma~\ref{l:Vdecomp}. If $m=1$ then we are done, so suppose otherwise.

Each $B_i$ normalizes $T$, and hence permutes the spaces
$V_1,\ldots,V_m$. If some $B_i$ induced a non-trivial permutation of these
spaces, then $\mathrm{rank}(B_i-I_d)\geq n\geq 3$, 
contradicting the fact that
$\mathrm{rank}(B_i-I_d)\leq 2$.
Therefore, every $B_i$ stabilizes each $V_j$.

The transvection $B_r$ has moved line $\langle\mathbf{1}\rangle$. After relabelling
the spaces $V_1,\ldots,V_m$, we may suppose that
$\mathbf{1}\in V_1$.
Since every $B_i$ stabilizes $V_1$, we have
$\mathbf{1} B_i\in V_1$.
On the other hand, from the description of the matrices $B_i$ in~\eqref{eq:binary-Bi} and~\eqref{eq:binary-B0}, we obtain
$\mathbf{1} B_i=a_i$
for every $i\in\{0,\ldots,d-1\}$. Thus
$a_i\in V_1$,
for every $i$.
It follows that
$[V,B_i]=\langle a_i,\mathbf{1}\rangle\leq V_1$,
for every $i\in\{0,\ldots,d-1\}$. Since $B_i$ stabilizes each $V_j$, this
implies that $B_i$ fixes
$V_2\oplus\cdots\oplus V_m$
pointwise. Therefore,
\begin{equation}\label{eq:common-fixed-space-Bi}
V_2\oplus\cdots\oplus V_m
\leq
\bigcap_{i=0}^{d-1}{\bf C}_V(B_i).
\end{equation}

Let
$x=\sum_{j=1}^{d}x_je_j$
belong to ${\bf C}_V(B_i)$, for every $i\in\{0,\ldots,d-1\}$.
If $\mathrm{rank}(B_i-I_d)=1$ for some
$i\in\{1,\ldots,d-1\}$, then from~\eqref{eq:rank-one-ai} we have $a_i=\mathbf{1}$. Therefore, 
$$
0
=x(B_i-I_d)
=x\left(e_i^\top a_i+e_d^\top \mathbf{1}\right)
=x_i a_i+x_d\mathbf{1}
=(x_i+x_d)\mathbf{1}.
$$
Since $\mathbf{1}\neq 0$, it follows that
$x_i+x_d=0$, and hence
$x_i=x_d$.
If $\mathrm{rank}(B_i-I_d)=2$, then the linear independence of
$a_i$ and $\mathbf{1}$ and a computation similar to the one above gives
$x_i=x_d=0$.
Similarly, if $\mathrm{rank}(B_0-I_d)=1$, then
$\sum_{j=1}^{d-1}x_j=0$,
while, if $\mathrm{rank}(B_0-I_d)=2$, then
$x_d=0$ and  $\sum_{j=1}^{d}x_j=0$.

By assumption, at least one of the matrices $B_i-I_d$ has
rank $2$. The preceding equations imply first that $x_d=0$, and then that
$x_i=0$
for every $i\in\{1,\ldots,d-1\}$. Hence $x=0$. This has shown 
$$\bigcap_{i=0}^{d-1}{\bf C}_V(B_i)=0.$$
This contradicts~\eqref{eq:common-fixed-space-Bi}, since $m>1$. Thus, $T$ is irreducible on $V$, as was to be shown.
\end{proof} Since $T$ is generated by transvections, arguing as in Section~\ref{sec:subsub1} we obtain the example in~\eqref{part12} of Theorem~\ref{thrm:main3}.

\subsubsection{Suppose 
$\mathrm{rank}(B_i-I_d)=2$, for every $i\in\{0,\ldots,d-1\}$.}
Equivalently, from~\eqref{eq:rank-one-ai}, we have
$a_i\neq\mathbf{1}$ for every $i\in\{0,\ldots,d-1\}$.

Our first goal is to prove that $G_0$ contains a transvection. First, we shall set up some notation.

Set
$$
u_0=\mathbf{1}^\top ,
\qquad
u_i=e_i^\top 
\quad
\text{for every }i\in\{1,\ldots,d-1\}.
$$
Then, using this notation, from~\eqref{eq:binary-Bi} and~\eqref{eq:binary-B0}, we have the uniform notation
\begin{equation}\label{eq:binary-Bi-uniform}
B_i=I_d+u_ia_i+e_d^\top \mathbf{1}
\end{equation}
for every $i\in\{0,\ldots,d-1\}$. Observe that
$u_0+u_1+\cdots+u_{d-1}=e_d^\top $,
and hence $u_0,\ldots,u_{d-1}$ form a basis of the column-vector space.

For $i,j\in\{0,\ldots,d-1\}$, set
$s_{i,j}=a_iu_j\in\mathbb F_2$.
From~\eqref{eq:binary-aid} and~\eqref{eq:binary-a0d}, we have that $a_ie_d^\top =1$; therefore
\begin{equation}\label{eq:row-sum-sij}
\sum_{j=0}^{d-1}s_{i,j}
=
a_i\left(\sum_{j=0}^{d-1}u_j\right)
=
a_ie_d^\top 
=
1
\end{equation}
for every $i\in\{0,\ldots,d-1\}$. In fact, \begin{equation}\label{eq:productshortcuts}
     a_je_d^\top =\mathbf{1} u_j=\mathbf{1} e_d^\top =1\,\,\hbox{ for all $j\in\{0,\dots,d-1\}.$}
\end{equation}.

To prove that $G_0$ contains a transvection we shall split our analysis into three cases:\begin{enumerate}
    \item there exists some $i\in\{0,\dots,d-1\}$ such that $s_{i,i}=1$ or there exist $i,j$ distinct elements of $\{0,\dots,d-1\}$ with $a_i=a_j$;
    \item $s_{i,i}=0$ and each of the $a_i$ are distinct for all $i\in\{0,\dots,d-1\}$, and there exist distinct $i,j\in\{0,\dots,d-1\}$ such that $s_{i,j}=s_{j,i}$; or
    \item $s_{i,i}=0$, $a_i\ne a_j$, and $s_{i,j}+s_{j,i}=1$ for all distinct $i,j\in\{0,\dots,d-1\}.$
\end{enumerate}
\begin{lemma}\label{l:transvection1}
    Suppose that either there exists some $i\in\{0,\dots,d-1\}$ such that $s_{i,i}=1$, or that there exist distinct $i,j\in\{0,\dots,d-1\}$ such that $a_i=a_j$. Then $G_0$ contains a transvection.
\end{lemma}
\begin{proof}
Suppose first that $s_{i,i}=1$ for some $i$. From~\eqref{eq:productshortcuts} and~\eqref{eq:binary-Bi-uniform},
we deduce
$$
\begin{aligned}
B_i^2-I_d
&=(u_ia_i+e_d^\top \mathbf{1})^2\\
&=u_i(a_iu_i)a_i+u_i(a_ie_d^\top )\mathbf{1}
  +e_d^\top (\mathbf{1} u_i)a_i+e_d^\top (\mathbf{1} e_d^\top )\mathbf{1}\\
&=(u_i+e_d^\top )(a_i+\mathbf{1}).
\end{aligned}
$$
Both factors are non-zero since $a_i\ne \mathbf{1}$ by assumption and $u_i\ne e_d^\top $ by construction. Moreover,
$(a_i+\mathbf{1})(u_i+e_d^\top )
=
s_{i,i}+1+1+1
=
0$.
Thus $B_i^2$ is a transvection and we are done.


Suppose now that $i,j$ are distinct elements of $\{0,\dots,d-1\}$ with $a_i=a_j$. Then
$B_i-B_j=(u_i+u_j)a_i$
has rank $1$. Therefore,
$ \mathrm{rank}(B_iB_j^{-1}-I_d)=1$,
and $B_iB_j^{-1}$ is a transvection, hence the result.
\end{proof}
\begin{lemma}\label{l:transvections2}
    Suppose that $s_{i,i}=0$ for all $i\in\{0,\dots,d-1\}$ and that $a_0,\dots,a_{d-1}$ are pairwise distinct. If there exist distinct $i,j\in\{0,\dots,d-1\}$ such that $s_{i,j}=s_{j,i}$ then $G_0$ contains a transvection.
\end{lemma}
\begin{proof}
Write
$s_{i,j}=s_{j,i}=\varepsilon\in\mathbb{F}_2$. Using repeatedly~\eqref{eq:productshortcuts} and
the fact that $s_{i,i}=0$, we calculate
\begin{align*}
    B_j^{-1}&=I_d-(1+a_ju_j)e_d^\top \mathbf{1}-(1+\mathbf{1}e_d^\top)u_ja_j+\mathbf{1}u_je_d^\top a_j+a_j e_d^\top u_j\mathbf{1}\\&=I_d+e_d^\top\mathbf{1}+e_d^\top a_j+u_j\mathbf{1},
\end{align*}
so \begin{align*}
    B_iB_j^{-1}=&I_d+u_ia_i+e_d^\top\mathbf{1}+e_d^\top\mathbf{1}+u_ia_ie_d^\top\mathbf{1}+e_d^\top\mathbf{1}e_d^\top\mathbf{1}\\&+e_d^\top a_j+u_ia_ie_d^\top a_j+e_d^\top\mathbf{1}e_d^\top a_j+u_j\mathbf{1}+u_ia_iu_j\mathbf{1}+e_d^\top\mathbf{1}u_j\mathbf{1}\\=&I_d+u_ia_i+u_i\mathbf{1}+u_i a_j+u_j\mathbf{1}+\varepsilon u_i\mathbf{1}\\=&I_d+u_i(a_i+a_j+(1+\varepsilon)\mathbf{1})+u_j\mathbf{1},
\end{align*}

and thus \begin{align*}
    (B_iB_j^{-1})^2-I_d=&(u_i(a_i+a_j+(1+\varepsilon)\mathbf{1})+u_j\mathbf{1})^2\\=&(u_i(a_i+a_j+(1+\varepsilon)\mathbf{1}))^2+u_i(a_i+a_j+(1+\varepsilon)\mathbf{1})u_j\mathbf{1}\\&+u_j\mathbf{1}u_i(a_i+a_j+(1+\varepsilon)\mathbf{1})+u_j\mathbf{1}u_j\mathbf{1}\\=&u_i(a_i+a_j+(1+\varepsilon)\mathbf{1})+u_i\mathbf{1}+u_j(a_i+a_j+(1+\varepsilon)\mathbf{1})+u_j\mathbf{1}\\=&(u_i+u_j)(a_i+a_j+\varepsilon\mathbf{1}).
\end{align*}

The column vector $u_i+u_j$ is non-zero. The row vector $(a_i+a_j+\varepsilon\mathbf{1})$ is also
non-zero: this is clear when $\varepsilon=0$, since $a_i\neq a_j$, while, if
$\varepsilon=1$, then
$(a_i+a_j+\mathbf{1})e_d^\top =1$.
Furthermore,
$$
(a_i+a_j+\varepsilon\mathbf{1})(u_i+u_j)=0.
$$
Thus $(B_iB_j^{-1})^2$ is a transvection, hence the result.\end{proof}
\begin{lemma}\label{l:transvection3}
    Suppose that $s_{i,i}=0$ for all $i\in\{0,\dots,d-1\}$ and that $a_0,\dots,a_{d-1}$ are pairwise distinct. If $s_{i,j}+s_{j,i}=1$ for all $i,j\in\{0,\dots,d-1\}$ distinct, then $G_0$ contains a transvection.
\end{lemma}

\begin{proof}
Since $s_{i,j}+s_{j,i}=1$, for all distinct $i,j\in\{0,\ldots,d-1\}$ and $s_{i,i}=0$ for all $i$, the matrix
$$
S=(s_{i,j})_{0\leq i,j\leq d-1}
$$
is the adjacency matrix of a tournament $\Gamma$ on $d$ vertices. By
\eqref{eq:row-sum-sij}, every vertex of this tournament has odd out-degree.

Since
$$
\sum_{i=0}^{d-1}\sum_{j=0}^{d-1}s_{i,j}
=
\binom{d}{2},
$$
and the left-hand side is odd (being the sum of an odd number of odd numbers) we have that $d(d-1)/2$ is odd. Hence
$d\equiv 3\pmod 4$.

The tournament $\Gamma$ is not transitive, since a transitive tournament has a vertex
of out-degree $0$, whereas every vertex in $\Gamma$ has odd out-degree. Hence it contains
a directed triangle
$$
r\longrightarrow s\longrightarrow t\longrightarrow r.
$$

We use the fact that $d>3$. We claim that the directed triangle may be chosen
so that
\begin{equation}\label{eq:non-zero-triangle-vector}
\mathbf{1}+a_r+a_s+a_t\neq 0.
\end{equation}
Indeed, suppose that
$\mathbf{1}+a_r+a_s+a_t=0$.
For every $x\notin\{r,s,t\}$, post-multiplying this equality by $u_x$ gives
$$
(\mathbf{1}+a_r+a_s+a_t)u_x=\mathbf{1} u_x+a_ru_x+a_su_x+a_tu_x=1+s_{r,x}+s_{s,x}+s_{t,x}=0.
$$
Thus an odd number of the vertices $r,s,t$ dominate $x$ in the tournament $S$.
Suppose first that precisely one of $r,s,t$ dominates $x$. After cyclically
relabeling $r,s,t$, we may assume that
$$
r\longrightarrow x,
\qquad
x\longrightarrow s,
\qquad
x\longrightarrow t.
$$
Then
$$
x\longrightarrow t\longrightarrow r\longrightarrow x
$$
is a directed triangle. Moreover,
$$
(\mathbf{1}+a_x+a_t+a_r)u_s
=
1+s_{x,s}+s_{t,s}+s_{r,s}
=
1+1+0+1
=
1.
$$
Hence
$\mathbf{1}+a_x+a_t+a_r\neq 0$ and~\eqref{eq:non-zero-triangle-vector} is satisfied  by this new directed triangle.

We may therefore suppose that every vertex outside $\{r,s,t\}$ is dominated
by all three of $r,s,t$. If the subtournament induced on
$\{0,\ldots,d-1\}\setminus\{r,s,t\}$
contains a directed triangle with vertices $x,y,z$, then
$(\mathbf{1}+a_x+a_y+a_z)u_r=1$, because $x,y,z$ are dominated by $r,s,t$. Hence
$\mathbf{1}+a_x+a_y+a_z\neq 0$  and~\eqref{eq:non-zero-triangle-vector} is satisfied  by this new directed triangle.
On the other hand, if this subtournament is transitive, then it has a vertex
of out-degree $0$ within the subtournament. Since this vertex is also
dominated by $r,s,t$, it has out-degree $0$ in the whole tournament, a
contradiction. This proves the claim.

Choose a directed triangle
$$
r\longrightarrow s\longrightarrow t\longrightarrow r
$$
satisfying~\eqref{eq:non-zero-triangle-vector}.
A direct calculation, using
\eqref{eq:binary-Bi-uniform}
gives
\begin{align*}
    B_rB_sB_t^{-1}=&\,B_r(I_d+u_s(a_s+a_t)+u_t\mathbf{1})\\=&\,I_d+u_ra_r+e_d^\top\mathbf{1}+u_s(a_s+a_t)+u_ra_ru_s(a_s+a_t)+e_d^\top\mathbf{1}u_s(a_s+a_t)\\&+u_t\mathbf{1}+u_ra_ru_t\mathbf{1}+e_d^\top\mathbf{1}u_t\mathbf{1}\\=&\,I_d+(u_s+u_r+e_d^\top)(a_s+a_t)+u_ra_r+u_t\mathbf{1},
\end{align*}
so

\begin{align}\begin{split}\label{eq:Y-square-original-coordinates}
(B_rB_sB_t^{-1})^2-I_d
&=((u_s+u_r+e_d^\top)(a_s+a_t)+u_ra_r+u_t\mathbf{1})^2
\\&=
(e_d^\top+u_r+u_s+u_t)(\mathbf{1}+a_r+a_s+a_t).
\end{split}
\end{align}
Both factors on the right-hand side are non-zero. Indeed, the row vector is
non-zero by~\eqref{eq:non-zero-triangle-vector}, while the column vector is
non-zero since $d>3$ and $u_0,\ldots,u_{d-1}$ form a basis.
We now show that the product in the reverse order is zero. First,
$$
(\mathbf{1}+a_r+a_s+a_t)e_d^\top 
=
1+1+1+1
=
0.
$$
Moreover,
$\mathbf{1}(u_r+u_s+u_t)=1$,
and, since
$$
r\longrightarrow s\longrightarrow t\longrightarrow r,
$$
we have
$$
(a_r+a_s+a_t)(u_r+u_s+u_t)
=
s_{r,s}+s_{s,t}+s_{t,r}
=
1.
$$
Therefore,
$
(\mathbf{1}+a_r+a_s+a_t)(e_d^\top +u_r+u_s+u_t)=0$.
It follows from~\eqref{eq:Y-square-original-coordinates} that $(B_rB_sB_t^{-1})^2$ is a
transvection.
\end{proof}

Thus, in every case, $G_0$ contains a transvection.
$$
T:=\left\langle x\in G_0\mid \mathrm{rank}(x-I_d)=1\right\rangle\ne 1,$$ and $T\trianglelefteq G_0$. 
\begin{lemma}
    Let $T=\left\langle x\in G_0\mid \mathrm{rank}(x-I_d)=1\right\rangle.$ Then $T\ne 1$ is an irreducible subgroup of $G_0$ generated by transvections.
\end{lemma}
\begin{proof} By Lemmas~\ref{l:transvection1},~\ref{l:transvections2}, and~\ref{l:transvection3}, $T\ne 1$. Moreover, it is clear that $T\unlhd G_0$. Suppose that $T$ is not irreducible. Then by Lemma~\ref{l:Vdecomp} there exists a decomposition
$$
V=V_1\oplus\cdots\oplus V_m
$$
into irreducible $T$-modules of the same odd dimension $n\geq 3$, where
$m>1$, and $G_0$ permutes the spaces $V_1,\ldots,V_m$ transitively. Since
$\mathrm{rank}(B_i-I_d)=2<n$,
each $B_i$ stabilizes every $V_j$.
For $i\in\{1,\ldots,d-1\}$, the linear independence of $a_i$ and $\mathbf{1}$
gives
$$
{\bf C}_V(B_i)
=
\left\{
\sum_{j=1}^{d}x_je_j\in V
:
x_i=x_d=0
\right\}.
$$
Since every $B_i$ stabilizes each $V_j$, every ${\bf C}_V(B_i)$ is decomposable with
respect to $V=V_1\oplus\cdots\oplus V_m$, that is, ${\bf C}_V(B_i)=\bigoplus_j{\bf C}_{V_j}(B_i)$.

Now set
$F=\langle e_1,\ldots,e_{d-1}\rangle$.
We have
$$
F=\sum_{i=1}^{d-1}{\bf C}_V(B_i),
$$
and, for every $j\in\{1,\ldots,d-1\}$,
$$
\langle e_j\rangle
=
\bigcap_{\substack{1\leq i\leq d-1\\i\neq j}}{\bf C}_V(B_i).
$$
Thus $F$ and each line $\langle e_j\rangle$ are decomposable with respect to $V=V_1\oplus\cdots\oplus V_m$, that is, $F=\bigoplus_k(F\cap V_k)$ and $\langle e_j\rangle=\bigoplus_k(\langle e_j\rangle\cap V_k)$. In particular,
each $e_j$ belongs to exactly one of the spaces $V_1,\ldots,V_m$.

Set
$F_k=F\cap V_k$.
Then $
F=F_1\oplus\cdots\oplus F_m$.
Since $F$ has codimension $1$ in $V$, all but one of the spaces $F_k$ are
equal to $V_k$, while the remaining one has codimension $1$ in the
corresponding $V_k$. In particular,
$F_k\neq 0$, for every $k\in\{1,\ldots,m\}$. In fact, each $F_k$ contains at least one of the vectors
$e_1,\ldots,e_{d-1}$. Indeed, every $e_j$ belongs to one of the
summands $V_k$, and the vectors $e_1,\ldots,e_{d-1}$ form a basis of
$F$. More explicitly,
$$
F\cap V_k=
\left\langle
e_j\mid e_j\in V_k
\right\rangle.
$$

On the other hand, the linear independence of $a_0$ and $\mathbf{1}$ gives
$$
{\bf C}_V(B_0)
=
\left\{
\sum_{j=1}^{d-1}x_je_j\in F
 :\
\sum_{j=1}^{d-1}x_j=0
\right\}.
$$
Let $\varphi:F\longrightarrow\mathbb F_2$
be the linear map defined by
$\varphi(e_j)=1$, for every $j\in\{1,\ldots,d-1\}$.
Then  ${\bf C}_V(B_0)=\ker(\varphi)$.

Since $B_0$ stabilizes every $V_k$, the space ${\bf C}_V(B_0)$ is decomposable
with respect to
$F=F_1\oplus\cdots\oplus F_m$.
A decomposable subspace of codimension $1$ in $F$ contains $F_k$ for all
but one value of $k$. On the other hand, since each $F_k$ contains some $e_j$ we deduce that
$\varphi|_{F_k}\neq 0$ for all $k$.
Thus no $F_k$ is contained in $\ker(\varphi)$, a contradiction, whence $T$ is irreducible. \end{proof}

Applying McLaughlin's classification~\cite{McL69} as in Sections~\ref{sec:subsub1} and~\ref{sec:subsub2}, we deduce that
$G_0=\mathrm{GL}_d(2)$.
This completes the analysis of the case $v=e_d$ and $w=\sum_{i=1}^{d-1}e_i$.

\subsection{Case $v=e_d$ and $e_1,\ldots,e_{d-1},w$ form a basis}
Since
$e_1,\ldots,e_{d-1},w$ is a basis of $V$, we write
$$
w=e_d+\delta,
$$
where $0\neq\delta\in\langle e_1,\ldots,e_{d-1}\rangle$.
In particular, $\delta e_d^\top =0$.

For each $i\in\{0,\ldots,d\}$, choose an affine transformation
$g_i=(t_i,B_i)\in G$ witnessing the $d$-equivalence of $I$ and $J$ after
deleting the $i$-th coordinate. As before, we may take
$B_d=I_d$ and $t_d=0$.

For $i\in\{1,\ldots,d-1\}$, set
$a_i=e_i(B_i-I_d).$
Then $t_i=0$ and
$B_i=I_d+e_i^\top a_i+e_d^\top \delta$.
For $i=0$, set $a_0=t_0$. Since
$e_jB_0+t_0=e_j$, for every $j\in\{1,\ldots,d-1\}$,
and $e_dB_0+t_0=w$,
we deduce that
$B_0=I_d+\mathbf{1}^\top a_0+e_d^\top \delta$.

By setting
$$
u_0=\mathbf{1}^\top ,
\qquad
u_i=e_i^\top 
\quad
\text{for every }i\in\{1,\ldots,d-1\},
$$
we may describe uniformly the matrices $B_i$ as
\begin{equation}\label{eq:basis-Bi-uniform}
B_i=I_d+u_ia_i+e_d^\top \delta,
\end{equation}
for every $i\in\{0,\ldots,d-1\}$. Notice that
$u_0,\ldots,u_{d-1}$ form a basis of the column-vector space. 

The following observations will be helpful.
\begin{lemma}\label{l:ainot0}
    Let $i\in\{0,\dots,d-1\}$. Then $a_i\ne 0$ and $\delta\in\mathrm{Im}(B_i-I_d).$
\end{lemma}
\begin{proof}
Suppose first that some $a_i=0$. Then $(t_i,B_i)=(0,I_d+e_d^\top\delta)$, but
$B_i=I_d+e_d^\top \delta$ fixes
$e_j$ for all $j\in\{1,\dots,d-1\}$, and maps $e_d$ to
$e_d+\delta=w$.
Thus the affine transformation $g_i=(t_i,B_i)$ maps the whole tuple
$I$ to the whole tuple $J$, contrary to our assumption.

For the second claim, when
$i\in\{1,\ldots,d-1\}$, this follows from
$e_d(B_i-I_d)=\delta$.
For $i=0$, observe that
$(e_d+e_1)u_0=0$ and $(e_d+e_1)e_d^\top =1$,
whence $(e_d+e_1)(B_0-I_d)=\delta$.
\end{proof}
\subsubsection{Suppose $\mathrm{rank}(B_r-I_d)=1$, for some $r\in\{0,\ldots,d-1\}$.}\label{sec:subsub4} 

\begin{lemma}
    Let
$
T=
\left\langle
x\in G_0\mid \mathrm{rank}(x-I_d)=1
\right\rangle.
$
Then $1\ne T\unlhd G_0$ and $T$ is irreducible. 
\end{lemma}
\begin{proof}
Suppose, for a contradiction, that $T$ is reducible. By Lemma~\ref{l:Vdecomp} there exists a decomposition
$$
V=V_1\oplus\cdots\oplus V_m
$$
into irreducible $T$-modules of the same dimension $n\geq 2$, where $m>1$,
and $G_0$ permutes the spaces $V_1,\ldots,V_m$ transitively.

For every $i\in\{0,\ldots,d-1\}$, equation
\eqref{eq:basis-Bi-uniform} gives
$\mathrm{rank}(B_i-I_d)\leq2$.
We first prove that every $B_i$ stabilizes every $V_k$.

Suppose first that $n>2$. If $B_i$ maps $V_k$ to a distinct
summand $V_\ell$, then the restriction
$(B_i-I_d)|_{V_k}$
is injective. Indeed, if $x\in V_k$ and $x(B_i-I_d)=0$, then
$x=xB_i\in V_k\cap V_\ell=0$.
Therefore, $
\mathrm{rank}(B_i-I_d)\geq\dim V_k=n>2$,
a contradiction. Thus every $B_i$ stabilizes every $V_k$ when
$n\geq3$.

It remains to consider the case $n=2$. Since $B_r$ is a transvection, its moved space has dimension $1$.
As $\delta$ is a non-zero element of this moved space by Lemma~\ref{l:ainot0}, we have
$\mathrm{Im}(B_r-I_d)=\langle\delta\rangle$.
Now $B_r\in T$, so $B_r$ stabilizes each $V_k$. Since $B_r$ stabilizes every $V_k$, the space
$$\mathrm{Im}(B_r-I_d)=\bigoplus_{k}\mathrm{Im}((B_r-I_d)_{V_k})$$
is decomposable. As this space equals the line $\langle\delta\rangle$, there
exists a unique summand, say $W$, such that
$\delta\in W$.
Fix $i\in\{0,\ldots,d-1\}$. Suppose first that $B_i$ maps $W$ to a
distinct summand $W'$. Then
$(B_i-I_d)|_W$
is injective, and hence its image has dimension $2$. Since
$\mathrm{rank}(B_i-I_d)\leq2$, we deduce
$$
\mathrm{Im}(B_i-I_d)
=
\mathrm{Im}\bigl((B_i-I_d)|_W\bigr).
$$
However, if $0\neq x\in W$, then $x(B_i-I_d)=xB_i+x$
has a non-zero component in $W'$. Therefore,
$$
\mathrm{Im}\bigl((B_i-I_d)|_W\bigr)\cap W=0.
$$
This contradicts $0\neq\delta\in
\mathrm{Im}(B_i-I_d)\cap W$.
Thus $B_i$ stabilizes $W$.

Suppose now that $B_i$ maps some summand $V_k$ to a distinct
summand $V_\ell$. Since $B_i$ stabilizes $W$, neither $V_k$ nor
$V_\ell$ is equal to $W$. As above,
$(B_i-I_d)|_{V_k}$
is injective, so its image has dimension $2$ and is contained in
$V_k\oplus V_\ell$. In particular, this image is disjoint from $W$.
Since $0\neq\delta\in
\mathrm{Im}(B_i-I_d)\cap W$,
we deduce that
\[
\mathrm{rank}(B_i-I_d)\geq3,
\]
a contradiction. Therefore, $B_i$ stabilizes every $V_k$ also when
$n=2$.

We have proved that every $B_i$ stabilizes each summand $V_k$.
Therefore,
\[
{\bf C}_V(B_i)
=
\bigoplus_{k=1}^m{\bf C}_{V_k}(B_i)
\]
for every $i\in\{0,\ldots,d-1\}$.

Since $u_i$ and $e_d^\top $ are linearly independent, from~\eqref{eq:basis-Bi-uniform}, we have
$$
\mathrm{rank}(B_i-I_d)
=
\dim\langle a_i,\delta\rangle.
$$
It follows that $
\mathrm{rank}(B_i-I_d)=1$
if and only if $a_i$ and $\delta$ are linearly dependent. Since both
vectors are non-zero by Lemma~\ref{l:ainot0} and the field is $\mathbb F_2$, this is equivalent
to $a_i=\delta$.

For $i\in\{1,\ldots,d-1\}$ and $x=\sum_{j=1}^dx_je_j$, we have
$x(B_i-I_d)=x_i a_i+x_d\delta.$
If $\mathrm{rank}(B_i-I_d)=2$, then $a_i$ and $\delta$ are
linearly independent, and hence
$$
{\bf C}_V(B_i)
=
\left\{
\sum_{j=1}^d x_je_j\in V
\mathrel{}\middle|\mathrel{}
x_i=x_d=0
\right\}.
$$
If $\mathrm{rank}(B_i-I_d)=1$, then $a_i=\delta$,
$B_i-I_d=(e_i^\top +e_d^\top )\delta$ and 
$$
{\bf C}_V(B_i)
=
\left\{
\sum_{j=1}^d x_je_j\in V
\mathrel{}\middle|\mathrel{}
x_i=x_d
\right\}.
$$

Similarly, for $i=0$, if $\mathrm{rank}(B_0-I_d)=2$, then
$a_0$ and $\delta$ are linearly independent. Thus
$$
{\bf C}_V(B_0)
=
\left\{
\sum_{j=1}^d x_je_j\in V
\mathrel{}\middle|\mathrel{}
x_d=0,\quad
\sum_{j=1}^{d-1}x_j=0
\right\}.
$$
If $\mathrm{rank}(B_0-I_d)=1$, then $a_0=\delta$ and $B_0-I_d=(u_0+e_d^\top )\delta$.
Since $x(u_0+e_d^\top )=\sum_{j=1}^{d-1}x_j$,
we obtain
$$
{\bf C}_V(B_0)
=
\left\{
\sum_{j=1}^d x_je_j\in V
\mathrel{}\middle|\mathrel{}
\sum_{j=1}^{d-1}x_j=0
\right\}.
$$

Set
$$
Q=
\left\{
i\in\{1,\ldots,d-1\}
\mid
\mathrm{rank}(B_i-I_d)=2
\right\}.
$$
We claim that $Q\neq\varnothing$. Indeed, suppose that
$\mathrm{rank}(B_i-I_d)=1$. Then $
B_i=I_d+(e_i^\top +e_d^\top )\delta$ is a transvection. Therefore $\delta(e_i^\top +e_d^\top )=0$. Since $\delta e_d^\top =0$, 
we have $\delta_i=0$.
If $Q$ were empty, then $\delta_i=0$ for every
$i\in\{1,\ldots,d-1\}$, contrary to $\delta\neq0$.

Set
$$
U=\langle e_1,\ldots,e_{d-1}\rangle.
$$
We now prove that $U$ is decomposable with respect to
$V=V_1\oplus\cdots\oplus V_m$, that is $U=\bigoplus_k U\cap V_k$

Suppose first that $Q$ contains two distinct indices $i$ and $j$.
Then ${\bf C}_V(B_i)
$
and
${\bf C}_V(B_j)$ are distinct hyperplanes of $U$, and therefore
\begin{align*}
U={\bf C}_V(B_i)+{\bf C}_V(B_j)&=\left(\bigoplus_{k=1}^m{\bf C}_{V_k}(B_i)\right)+\left(\bigoplus_{k=1}^m{\bf C}_{V_k}(B_j)\right)\\&=\bigoplus_{k=1}^m({\bf C}_{V_k}(B_i)+{\bf C}_{V_k}(B_j))\\&=\bigoplus_{k=1}^m((V_k\cap{\bf C}_{V}(B_i))+(V_k\cap{\bf C}_{V}(B_j))).\end{align*} Write $U_k=(V_k\cap{\bf C}_{V}(B_i))+(V_k\cap{\bf C}_{V}(B_j))\leq U\cap V_k$. If $u\in U\cap V_k$ then $u=u_1+\cdots+u_m$ where $u_\ell\in U_\ell\leq V_\ell$ for all $\ell\in\{1,\dots,m\}$. But $u\in V_k$, so $u_\ell=0$ for all $\ell\in\{1,\dots,m\}\setminus\{k\}$, whence $u\in U_k$. Therefore, $U=\bigoplus_k U\cap V_k$, as desired.

Suppose now that $Q=\{q\}$. For every $i\neq q$, the matrix
$B_i-I_d$ has rank $1$, so $\delta_i=0$ by the calculation above.
Since $\delta\neq0$, we have $\delta=e_q$.
We claim that $B_0-I_d$ has rank $2$. If it had rank $1$, then
$B_0=I_d+(u_0+e_d^\top )\delta$ is a transvection and hence
$0=\delta(u_0+e_d^\top )=\sum_{j=1}^{d-1}\delta_j$,
a contradiction.

Since some $B_r$ is a transvection, while $B_q-I_d$ and
$B_0-I_d$ both have rank $2$, we must have $d\geq3$. Now
$$
{\bf C}_V(B_q)
=
\left\{
x\in U\mid x_q=0
\right\},
$$
whereas
$$
{\bf C}_V(B_0)
=
\left\{
x\in U
\mathrel{}\middle|\mathrel{}
\sum_{j=1}^{d-1}x_j=0
\right\}.
$$
From this description, we see that these are subspaces of $V$ contained in $U$. Hence, they are distinct hyperplanes of $U$, and hence
$U={\bf C}_V(B_q)+{\bf C}_V(B_0)$.
Thus $U$ is decomposable in this case as well.

For each $i\in\{1,\ldots,d-1\}$,
whether $B_i-I_d$ has rank $1$ or rank $2$, the descriptions above
yield
$$
{\bf C}_U(B_i)
=
\left\{
\sum_{j=1}^{d-1}x_je_j\in U
\mathrel{}\middle|\mathrel{}
x_i=0
\right\}.
$$
Since both $U$ and ${\bf C}_V(B_i)$ are decomposable, each ${\bf C}_U(B_i)$ is
decomposable.

For every $j\in\{1,\ldots,d-1\}$, we have
$$
\langle e_j\rangle
=
\bigcap_{\substack{1\leq i\leq d-1\\ i\neq j}}{\bf C}_U(B_i).$$
Therefore, every line $\langle e_j\rangle$ is decomposable. Since it
has dimension $1$, it is contained in one of the summands
$V_1,\ldots,V_m$. Thus every $e_j$ belongs to one of the spaces
$V_1,\ldots,V_m$.

Since $U$ is decomposable, we have $U=(U\cap V_1)\oplus\cdots\oplus (U\cap V_m)$ and hence
$$1=\mathrm{codim}_V(U)
=
\sum_{k=1}^m\mathrm{codim}_{V_k}(U\cap V_k).
$$
It follows that all but one of the spaces $U\cap V_k$ are equal to $V_k$,
while the remaining one has codimension $1$ in the corresponding
$V_k$. Since $n\geq2$, we deduce that $U\cap V_k\neq0$,
for every $k$.

Since every $e_j$ belongs to one of the
summands $V_k$, and the vectors $e_1,\ldots,e_{d-1}$ form a basis of
$U$, we deduce that each $U\cap V_k$ contains at least one of the vectors
$e_1,\ldots,e_{d-1}$.

In both the rank-one and rank-two cases for $B_0-I_d$, we have
$$
{\bf C}_U(B_0)
=
\left\{
\sum_{j=1}^{d-1}x_je_j\in U
\mathrel{}\middle|\mathrel{}
\sum_{j=1}^{d-1}x_j=0
\right\}.
$$
Let  $\varphi:U\longrightarrow\mathbb F_2$ be the linear map defined 
by $\varphi(e_j)=1$,
for every $j\in\{1,\ldots,d-1\}$. Then  ${\bf C}_U(B_0)=\ker(\varphi)$.

Since both $U$ and ${\bf C}_V(B_0)$ are decomposable, the subspace
$U\cap {\bf C}_V(B_0)={\bf C}_U(B_0)$ is decomposable with respect to
$U=(U\cap V_1)\oplus\cdots\oplus (U\cap V_m)$.
Thus
$$
{\bf C}_U(B_0)=
\bigoplus_{k=1}^m({\bf C}_U(B_0)\cap V_k).
$$
Since ${\bf C}_U(B_0)$ has codimension $1$ in $U$, we have
$$
1=
\sum_{k=1}^m
\mathrm{codim}_{U\cap V_k}({\bf C}_U(B_0)\cap V_k).
$$
Therefore, ${\bf C}_U(B_0)$ contains $U\cap V_k$ for all but one value of $k$.

On the other hand, every $U\cap V_k$ contains some $e_j$, and
$\varphi(e_j)=1$.
Hence $
\varphi|_{U\cap V_k}\neq0$,
and therefore $U\cap V_k\not\leq\ker(\varphi)={\bf C}_U(B_0)$
for every $k$. Since $m>1$, this is a contradiction.

Therefore, $T$ is irreducible. \end{proof}

Applying McLaughlin's classification~\cite{McL69}, one of the following holds:
\begin{enumerate}
    \item $T=\mathrm{SL}_d(2)$;
    \item $d$ is even and $T=\mathrm{Sp}_d(2)$;
    \item $d$ is even and $T=\mathrm{O}_d^{\pm}(2)$;
    \item $d$ is even and $T$ is the natural image of
    $\mathrm{Sym}(d+1)$ or $\mathrm{Sym}(d+2)$ on its fully deleted
    permutation module.
\end{enumerate}
Now, Theorem~\ref{thrm:main3} immediately follows (in this case) from the fact that $T\le G_0={\bf N}_{G_0}(T)\le{\bf N}_{\mathrm{GL}_d(2)}(T)=T$ (see~\cite[Chapter~2]{KleidmanLiebeck}) and from Section~\ref{sec:someq=2}.

\subsubsection{Suppose that
$\mathrm{rank}(B_i-I_d)=2$ for every
$i\in\{0,\ldots,d-1\}$.}
\label{sec:basis-rank-two-complete}

Since the column vectors $u_i$ and $e_d^\top $ are linearly independent,~\eqref{eq:basis-Bi-uniform} shows that $a_i$ and $\delta$ are
linearly independent, for every $i\in\{0,\ldots,d-1\}$. Set
\[
U=\langle e_1,\ldots,e_{d-1}\rangle.
\]
For $i\in\{1,\ldots,d-1\}$, we have
$x(B_i-I_d)=x_i a_i+x_d\delta$, and hence
\begin{equation}\label{eq:rank-two-fixed-i-complete}
{\bf C}_V(B_i)
=
\left\{
\sum_{j=1}^d x_je_j\in V
\mathrel{}\middle|\mathrel{}
x_i=x_d=0
\right\}.
\end{equation}
Similarly,
\begin{equation}\label{eq:rank-two-fixed-zero-complete}
{\bf C}_V(B_0)
=
\left\{
\sum_{j=1}^d x_je_j\in V
\mathrel{}\middle|\mathrel{}
x_d=0,\quad
\sum_{j=1}^{d-1}x_j=0
\right\}.
\end{equation}
Thus all the spaces in
\eqref{eq:rank-two-fixed-i-complete} and
\eqref{eq:rank-two-fixed-zero-complete} are hyperplanes of $U$.
Moreover,
\begin{equation}\label{eq:intersection-Bi-zero-complete}
\bigcap_{i=0}^{d-1}{\bf C}_V(B_i)=0.
\end{equation}

Set
\begin{equation}\label{eq:C-complete}
C=I_d+e_d^\top \delta.
\end{equation}
Note that, since $\delta e_d^T=0$, $C$ is a transvection.
Then $C$ fixes $e_1,\ldots,e_{d-1}$ and maps $e_d$ to
$e_d+\delta=w$. Therefore $C$ maps the whole tuple $I$ to the whole
tuple $J$. Since $I\not\sim_{d+1}J$, we have
\begin{equation}\label{eq:C-not-G0-complete}
C\notin G_0.
\end{equation}
We remind the reader that, since Theorem~\ref{thrm:main3} has been verified by computer for $d\leq 8$, we are assuming that $d\geq 9$. We first show that $G_0$ contains no transvection.
\begin{lemma}\label{l:G0notransvection}
    The group $G_0$ contains no transvection.
\end{lemma}
\begin{proof} Suppose, for contradiction, that $G_0$ contains a transvection.
Set
\[
T=
\left\langle
x\in G_0\mid \mathrm{rank}(x-I_d)=1
\right\rangle.
\]
Then $1\neq T\unlhd G_0$. We claim first that $T$ is irreducible.

Suppose otherwise. By Lemma~\ref{l:Vdecomp}, there is a decomposition
\begin{equation}\label{eq:T-decomposition-complete}
V=V_1\oplus\cdots\oplus V_m
\end{equation}
for some $m>1$, where the $V_k$ are irreducible $T$-modules of the same
dimension $n\geq 2$, and $G_0$ permutes them transitively.

Every $B_i$ normalizes $T$ and therefore permutes the summands in
\eqref{eq:T-decomposition-complete}. If $n\geq3$ and $B_i$ maps
$V_k$ to a distinct summand, then
$(B_i-I_d)|_{V_k}$ is injective, so
\[
\mathrm{rank}(B_i-I_d)\geq n>2,
\]
a contradiction. Hence every $B_i$ stabilizes every $V_k$ when
$n\geq3$.

Suppose that $n=2$. A non-trivial cycle of length $\ell$ on the blocks
contributes at least $2(\ell-1)$ to the rank of $B_i-I_d$. Thus the permutation
induced by a matrix $B_i$ on the summands is either trivial or a single transposition. 
Assume that some $B_r$ interchanges two summands; without loss of generality assume $B_r$ interchanges $V_1$ and $V_2$.
The restriction $(B_r-I_d)|_{V_1}$ is injective and has two-dimensional
image. Since $\mathrm{rank}(B_r-I_d)=2$, this is the whole moved space.
Every non-zero vector in this image has non-zero components in both
$V_1$ and $V_2$. Thus Lemma~\ref{l:ainot0} implies that
$\delta$ has non-zero components in precisely $V_1$ and $V_2$.

If another $B_i$ interchanges two summands, the same argument and
Lemma~\ref{l:ainot0} show that it interchanges $V_1$ and
$V_2$. On the other hand, if $B_i$ stabilizes every summand, then the restrictions of $B_i-I_d$ to both $V_1$ and
$V_2$ are non-zero since $\delta\in\mathrm{Im}(B_i-I_d)$ and the $V_1$- and $V_2$-components
of $\delta$ are non-zero. Since $\mathrm{rank}(B_i-I_d)=2$, it follows that $B_i$ fixes every other
summand pointwise, and similarly for a matrix $B_i$ interchanging $V_1$ and $V_2$. Hence every summand different from $V_1$ and $V_2$ is contained in the intersection in~\eqref{eq:intersection-Bi-zero-complete}. It follows that $m=2$ and $d=4$, a contradiction to the assumption that $d\geq 9$.

We may therefore assume that every $B_i$ stabilizes every summand in
\eqref{eq:T-decomposition-complete} and that $n\geq 2$. The fixed spaces in
\eqref{eq:rank-two-fixed-i-complete} and
\eqref{eq:rank-two-fixed-zero-complete} are then decomposable with
respect to~\eqref{eq:T-decomposition-complete}. Since $d>8$, the elements $B_1$ and $B_2$ exist and satisfy $U={\bf C}_V(B_1)+{\bf C}_V(B_2)$, and hence $U$ is decomposable. Also, for every
$j\in\{1,\ldots,d-1\}$,
\[
\langle e_j\rangle
=
\bigcap_{\substack{1\leq i\leq d-1\\i\neq j}}
{\bf C}_V(B_i),
\]
so every line $\langle e_j\rangle$ is decomposable and is contained in
one summand.

Put $U_k=U\cap V_k$. Since $U$ has codimension $1$ in $V$, every
$U_k$ is non-zero. Since the coordinate lines span $U$ and each of
them belongs to one summand, every $U_k$ contains some $e_j$.

Define
\[
\varphi:U\longrightarrow\mathbb F_2,
\qquad
\varphi\left(\sum_{j=1}^{d-1}x_je_j\right)
=
\sum_{j=1}^{d-1}x_j.
\]
Then ${\bf C}_V(B_0)=\ker(\varphi)$. This is a decomposable hyperplane
of the decomposable space
$U=U_1\oplus\cdots\oplus U_m$. Therefore it contains $U_k$ for all
but at most one value of $k$. On the other hand, every $U_k$ contains an
$e_j$, and $\varphi(e_j)=1$. Thus no $U_k$ is contained in
$\ker(\varphi)$, a contradiction.

Therefore $T$ is irreducible. Finally, applying McLaughlin's classification, and using the exclusions already proved at the end of Section~\ref{sec:subsub4}, gives $T=\mathrm{SL}_d(2)=\mathrm{GL}_d(2)$.
Thus $G_0=\mathrm{GL}_d(2)$, contrary to
\eqref{eq:C-not-G0-complete}, hence the result.\end{proof}

For every $i\in\{0,\ldots,d-1\}$, set
\[
b_i=a_i+(a_ie_d^\top )\delta
\qquad\text{and}\qquad
X_i=B_iC.
\]
Using $\delta e_d^\top =0$, we obtain
$$X_i=I_d+u_ib_i.$$
Since $\det(X_i)=1$, we deduce
$b_iu_i=0$. Moreover, $b_i\neq0$ and $b_i\neq\delta$, since either
equality would imply that $a_i$ and $\delta$ are linearly dependent.
Thus every $X_i$ is a transvection, and $B_i=X_iC$. Moreover, since $X_i$ is a transvection, $X_i^2=I_d$, so $CB_iC=B_i^{-1}$.

Set $H=\langle B_0,\ldots,B_{d-1}\rangle$. Then $H\leq G_0$, so $H$ contains no transvection by Lemma~\ref{l:G0notransvection}. Put
\[
\mathcal I=\{\infty,0,\ldots,d-1\},
\]
and define
\[
c_\infty=e_d^\top ,\quad v_\infty=\delta,\quad Y_\infty=C,
\]
and
\[
c_i=u_i,\quad v_i=b_i,\quad Y_i=X_i
\quad
(i\in\{0,\ldots,d-1\}).
\]
Thus $Y_i=I_d+c_iv_i$ is a transvection for every
$i\in\mathcal I$.

\begin{lemma}\label{l:symmetric-pairing-complete}
For all $i,j\in\mathcal I$, we have $v_ic_j=v_jc_i$.
Moreover, every $B_i$ has order $2$ or $3$.
\end{lemma}

\begin{proof}
Let $i\neq j$. The product $Y_iY_j$ belongs to $H$. Indeed, if
$j=\infty$ then $Y_iY_j=B_i$, if $i=\infty$ then $Y_iY_j=B_j^{-1}$, and if $i,j\neq\infty$, then $X_iX_j=B_iB_j^{-1}\in H$.
Set $\lambda=v_ic_j$ and $\mu=v_jc_i$. If
$(\lambda,\mu)=(1,0)$, a direct multiplication gives
\[
(Y_iY_j)^2=I_d+c_iv_j,
\]
which is a transvection in $H$. If $(\lambda,\mu)=(0,1)$, then
\[
(Y_iY_j)^2=I_d+c_jv_i
\]
is a transvection in $H$. Both options contradict
Lemma~\ref{l:G0notransvection}, whence $\lambda=\mu$ and the first part of the statement is now proven.

Now $B_i=X_iC=Y_iY_\infty$. If
$v_ic_\infty=v_\infty c_i=0$, then the two transvections commute and their
nonidentity product has order $2$. If
$v_ic_\infty=v_\infty c_i=1$, then a direct computation shows that the product $Y_iY_\infty$ has order $3$.
\end{proof}

Now, let $N=\langle H^g\mid g\in G_0\rangle$.
Then $1\neq N\unlhd G_0$. Furthermore, $N$ is generated by the conjugates of the matrices
$B_0,\ldots,B_{d-1}$. Thus it is generated by bireflections (that is, elements $B$ satisfying $\mathrm{rank}(B-I)=2$) of order
$2$ or $3$, and it contains no transvection. Recall that we are assuming $d>8$.

\begin{lemma}\label{l:N-absolute-complete}
The group $N$ is absolutely irreducible on $V$.
\end{lemma}

\begin{proof}
First we prove irreducibility. Since $N\unlhd G_0$, the $N$-socle of
$V$ is $G_0$-invariant. Hence $V$ is completely reducible as an
$N$-module. Suppose that
\[
V=V_1\oplus\cdots\oplus V_m
\]
is a decomposition into irreducible $N$-modules with $m>1$.
The fixed-point space ${\bf C}_V(N)$ is $G_0$-invariant. Since
$N\neq1$ and $G_0$ is irreducible, ${\bf C}_V(N)=0$. Therefore no
$V_k$ is a trivial $N$-module; in particular,
$\dim V_k\geq2$, because $\mathrm{GL}_1(2)=1$.

Every $B_i$ belongs to $N$ and hence stabilizes every $V_k$. Therefore
all the fixed spaces in
\eqref{eq:rank-two-fixed-i-complete} and
\eqref{eq:rank-two-fixed-zero-complete} are decomposable. The
fixed-space argument used above now applies without any permutation of
summands: $U={\bf C}_V(B_1)+{\bf C}_V(B_2)$ is decomposable, every
coordinate line $\langle e_j\rangle$ is decomposable, every non-zero
space $U\cap V_k$ contains some $e_j$, and the decomposable hyperplane
${\bf C}_V(B_0)=\ker(\varphi)$ cannot contain all but one of the spaces
$U\cap V_k$. This contradiction proves that $N$ is irreducible.

Let $E=\mathrm{End}_N(V)$.
By Schur's lemma, $E$ is a finite division ring, and hence
$E\cong\mathbb F_{2^s}$ for some $s\geq1$. Every
${\bf C}_V(B_i)$ is an $E$-subspace, because every element of $E$
commutes with $B_i$. Since $U={\bf C}_V(B_1)+{\bf C}_V(B_2)$,
the space $U$ is also an $E$-subspace. It follows that
$s$ divides $\dim_{\mathbb F_2}V=d$ and $s$ divides $\dim_{\mathbb F_2}U=d-1$.
Thus $s=1$, which is equivalent to absolute
irreducibility.
\end{proof}

Let
\[
k=\overline{\mathbb F}_2
\qquad\text{and}\qquad
\overline V=k\otimes_{\mathbb F_2}V.
\]
We identify $N$ with its faithful image in $\mathrm{GL}(\overline V)$. By
Lemma~\ref{l:N-absolute-complete}, the action of $N$ on $\overline{V}$ is irreducible.
Extension of scalars preserves the rank of $x-I_d$ for every
$x\in N$. Thus $N$ is still generated by bireflections of order $2$
or $3$, and it still contains no transvection.

\begin{lemma}\label{l:N-primitive-complete}
The group $N$ is primitive on $\overline V$.
\end{lemma}

\begin{proof}
Suppose that $\overline V=W_1\oplus\cdots\oplus W_m$
is a non-trivial system of imprimitivity for $N$, and put
$r=\dim_k W_j$.

Assume first that $r\geq3$. If some $B_i$ maps $W_j$ to a distinct
block, then $(B_i-I_d)|_{W_j}$ is injective, and hence
$\mathrm{rank}(B_i-I_d)\geq r>2$. Thus every $B_i$ stabilizes every
block. Their fixed spaces are therefore decomposable, and the
fixed-space argument in the proof of
Lemma~\ref{l:N-absolute-complete} -- now over $k$ -- gives a contradiction.

Suppose that $r=2$. A non-trivial cycle of length $\ell$ on the blocks
contributes at least $2(\ell-1)$ to the rank of $B_i-I_d$. Thus the permutation
induced by a matrix $B_i$ on the blocks is either trivial or a single transposition.
If every $B_i$ stabilizes every block, the same fixed-space argument
again gives a contradiction. Hence some $B_r$ interchanges two blocks; without loss of generality assume that $B_r$ interchanges $W_1$ and $W_2$. The restriction
$(B_r-I_d)|_{W_1}$ is injective and has dimension $2$, so its image is the
whole moved space of $B_r$. Every non-zero vector in that image has
non-zero components in both $W_1$ and $W_2$. Since
$\delta\in\mathrm{Im}(B_r-I_d)$, the vector $\delta$ has non-zero
components in precisely these two blocks.

The same reasoning shows that every $B_i$ which interchanges blocks
must interchange $W_1$ and $W_2$. If $B_i$ stabilizes
every block, then
$$
\mathrm{Im}(B_i-I_d)
=
\bigoplus_W
\mathrm{Im}\bigl((B_i-I_d)|_W\bigr).
$$
Since $\delta\in\mathrm{Im}(B_i-I_d)$ has non-zero components in
both $W_1$ and $W_2$, the restrictions of $B_i-I_d$ to $W_1$ and to $W_2$ are
both non-zero. Each of them therefore contributes at least $1$ to $\mathrm{rank}(B_i-I_d)$. Since $
\mathrm{rank}(B_i-I_d)=2$, we deduce that $(B_i-I_d)|_W=0$,
for every block $W_j$ with $j>2$. Thus $B_i$ fixes every such block
pointwise. Similarly, a matrix $B_i$ interchanging $W_1$ and $W_2$ also fixes every other block pointwise. Therefore
every other block lies in
$\bigcap_i{\bf C}_{\overline V}(B_i)=0$. Hence $m=2$ and $d=4$, which contradicts our assumption $d>8$.

It remains to consider $r=1$. Let $\Omega$ be the set of the
one-dimensional blocks, let
\[
\pi:N\longrightarrow P\leq\mathrm{Sym}(\Omega)
\]
be the induced permutation action, and set $D=\ker(\pi)$. The group
$P$ is primitive: a non-trivial block system for $P$ would give a
system of imprimitivity for $N$ with blocks of dimension greater than
$1$, which has already been excluded.

Let $x$ be one of the generating conjugates of a matrix $B_i$. If
$|x|=2$, then each transposition in the permutation induced by $x$
contributes 1 to $\mathrm{rank}(x-I_d)$, while a fixed
one-dimensional block contributes nothing. Hence $\pi(x)$ is a
product of two disjoint transpositions. If $|x|=3$, then either
$\pi(x)$ is a $3$-cycle, or $\pi(x)=1$ and $x$ is diagonal. It follows
that the non-trivial group $P$ (if $P=1$, then every one-dimensional block $W_\omega$ is $N$-invariant, contradicting the irreducibility of $N$ on 
$\overline{V}$) contains either a $3$-cycle or a double
transposition. By the classical Jordan Theorems for primitive
permutation groups, and since $d>8$,
\[
\mathrm{Alt}(\Omega)\leq P;
\]
see, for example, \cite[Section~3.3]{DixonMortimer}. Thus
$P=\mathrm{Alt}(\Omega)$ or $P=\mathrm{Sym}(\Omega)$.

Thus
$$\overline V=\bigoplus_{\omega\in\Omega}W_\omega,
\qquad
\dim_k(W_\omega)=1,$$
where $W_\omega$ denotes the one-dimensional block indexed by
$\omega$. Since $D$ is the kernel of the action on $\Omega$, every
element of $D$ stabilizes each line $W_\omega$. Consequently, for
every $\omega\in\Omega$ and every $x\in D$, there is a unique scalar
$\chi_\omega(x)\in k^\ast$ such that
\[
wx=\chi_\omega(x)w
\]
for every $w\in W_\omega$. The map
\[
\chi_\omega:D\longrightarrow k^\ast
\]
is a linear character, describing the action of $D$ on the
one-dimensional space $W_\omega$.

Suppose first that $D\neq1$. Equality of two such characters is a $P$-invariant
equivalence relation on $\Omega$. Since $P$ is primitive, either all
the characters are equal or they are pairwise distinct. In the first
case $D$ consists of scalar matrices. But every element of $D$ belongs
to $\mathrm{GL}_d(2)$, so the only scalar it can contain is the
identity, contrary to $D\neq1$. Hence the characters are pairwise
distinct, and the spaces $W_\omega$ are precisely the common
eigenspaces of $D$.

The Frobenius map
\[
\sigma:\overline V\longrightarrow\overline V,\qquad
(x_1,\ldots,x_d)\longmapsto(x_1^2,\ldots,x_d^2),
\]
commutes with $N$, and therefore permutes the common eigenspaces
$W_\omega$. The induced permutation centralizes $P$. Since
$P\geq\mathrm{Alt}(\Omega)$ and $|\Omega|=d>3$, its centralizer in
$\mathrm{Sym}(\Omega)$ is trivial. Thus every $W_\omega$ is
Frobenius-stable, and hence is defined over $\mathbb F_2$. We obtain
a decomposition of $V$ into one-dimensional $\mathbb F_2$-spaces
permuted by $N$. Since $\mathrm{GL}_1(2)=1$, the group acts by
permutation matrices and fixes the sum of a basis vector from each
line. This contradicts the irreducibility of $N$ on $V$.

Finally, suppose that $D=1$. Then $N\cong P$. The stabilizer of a line
$W_\omega$ acts on it through a linear character. If
$P=\mathrm{Alt}(d)$, this stabilizer is $\mathrm{Alt}(d-1)$ and is
perfect. If $P=\mathrm{Sym}(d)$, its abelianization has order $2$, but
$k^\ast$ has no element of order $2$. In either case the character
is trivial. Hence $\overline V$ is the ordinary permutation module of
$P$ on $\Omega$, which has a non-zero fixed vector. This again
contradicts irreducibility.
\end{proof}

\begin{lemma}\label{l:N-tensor-complete}
The group $N$ is tensor-indecomposable on $\overline V$.
\end{lemma}

\begin{proof}
We use Lemma~\ref{l:symmetric-pairing-complete}. 
Suppose that $N$ normalizes a non-trivial tensor decomposition
\begin{equation}\label{eq:tensor-decomposition-complete}
\overline V=V_1\otimes\cdots\otimes V_t,
\qquad
\dim V_j\geq2.
\end{equation}

We first show that no generating bireflection permutes the tensor
factors non-trivially. Let $x$ be such a generator and suppose first that it permutes the tensor factors non-trivially. If $|x|=2$ then it swaps two $r$-dimensional factors; without loss of generality assume that these are $V_1$ and $V_2$. Observe that the product of the dimensions
of the remaining factors is $d/r^2$. After changing bases in $V_1$ and $V_2$ and absorbing a scalar into the remaining tensor factor, the
action has the form $\tau\otimes y$, where $\tau$ is the swap on
$k^r\otimes k^r$ and $y$ is an involution on a $(d/r^2)$-dimensional
space. 

Abusing notation we identify $V_1$ and $V_2$ and choose a basis $\{e_1,\dots, e_r\}$ such that the swap $\tau$ satisfies $\tau(e_i\otimes e_j) = e_j\otimes e_i$ for all $i,j$. When $i\neq j$, this yields a Jordan block, $J_2$, on the space $\langle e_i\otimes e_j, e_j\otimes e_i\rangle$ corresponding to the matrix $\begin{pmatrix} & 1 \\ 1 & \end{pmatrix}$. Thus $\tau$ has
$r(r-1)/2$ Jordan blocks of size $2$. Since the rank of $J_2-I_2$ is equal to $1$, we  see that
\[
\mathrm{rank}(x-I_d)\geq \frac{d}{r^2}\cdot\frac{r(r-1)}2=\frac{d(1-\frac{1}{r})}{2}\geq d/4>2,\]
a contradiction.

If $|x|=3$ then $x$ induces a $3$-cycle
of $r$-dimensional factors; the cyclic permutation on
$(k^r)^{\otimes3}$ has eigenspace dimensions
\[
\frac{r^3+2r}{3},\qquad
\frac{r^3-r}{3},\qquad
\frac{r^3-r}{3}.
\]
Tensoring with any semisimple element of order dividing $3$ on the
remaining $\left(\frac{d}{r^3}\right)$-dimensional tensor factor cannot give a $1$-eigenspace
larger than $(d/r^3)(r^3+2r)/3$. Hence
\[
\mathrm{rank}(x-I_d)
\geq
\frac{d}{r^3}\cdot\frac{2(r^3-r)}3>2.
\]
Thus every generating bireflection stabilizes every tensor factor, and so we may absorb some factors and assume without loss of generality that $N$ normalizes a decomposition $V=V_1\otimes V_2$.

We next use the following elementary two-factor calculation. Let
$V_1$ and $V_2$ have dimensions $r,s\geq2$, and let
$x=a\otimes b$ have order $2$ or $3$. If both $a$ and $b$ are
projectively non-scalar, then
\begin{equation}\label{eq:tensor-rank-bound-complete}
\mathrm{rank}(x-I)\geq r+s-2.
\end{equation}
Indeed, for order $2$, after rescaling the two factors we may write
$a=I+P$ and $b=I+Q$, where $P^2=Q^2=0$. If
$p=\mathrm{rank}(P)$ and $q=\mathrm{rank}(Q)$, direct calculation gives
\[
\mathrm{rank}(a\otimes b-I)
=
ps+rq-2pq
\geq r+s-2.
\]
For order $3$, the two factors are semisimple. If their three
eigenvalue multiplicities are
$(r_0,r_1,r_2)$ and $(s_0,s_1,s_2)$, respectively, each triple has at
least two non-zero entries. The multiplicity of any prescribed
eigenvalue of $a\otimes b$ is at most
\[
(r-1)(s-1)+1.
\]
This proves~\eqref{eq:tensor-rank-bound-complete}.
Since $rs=d\geq 9$ it follows that $r+s-2\geq 4$, hence each generating bireflection is
projectively non-scalar on exactly one factor, without loss of generality suppose that $x$ is projectively non-scalar on $V_1$. After absorbing scalars,
$x=c\otimes I_s$, and hence
\[
2=\mathrm{rank}(x-I_d)
=
s\,\mathrm{rank}(c-I_r).
\]
Since the tensor decomposition is non-trivial, $s\geq2$, and therefore
$s=2$ and $r=d/2$. Since $d/2>2$, any generating bireflection which is projectively non-scalar on exactly one factor must act by scalar on $V_2$. Therefore, $N$ stabilizes every subspace
$V_1\otimes L$, where $L$ is a one-dimensional subspace of the
two-dimensional factor. This contradicts irreducibility.
\end{proof}

We may now apply the bireflection theorem of Guralnick and Saxl~\cite[Theorem~7.A]{GuralnickSaxl}. Recall that $k=\overline{\mathbb F}_2$ and $\overline{V}=k\otimes_{\mathbb F_2} V$.

\begin{lemma}\label{l:GS-consequence-complete}
One of the following holds:
\[
N\cong\Omega_d^+(2),\quad
N\cong\Omega_d^-(2),\quad
N\cong\mathrm{Alt}(d+1),\quad\text{or}\quad
N\cong\mathrm{Alt}(d+2),\quad
\]
in the corresponding natural or fully deleted representation.
In particular, $d$ is even.
\end{lemma}

\begin{proof}
By Lemmas~\ref{l:N-absolute-complete},
\ref{l:N-primitive-complete}, and~\ref{l:N-tensor-complete}, the
hypotheses of \cite[Theorem~7.A]{GuralnickSaxl} hold over
$k$. Since $N$ contains a bireflection,
the invariant $\nu_N(\overline V)$ in that theorem is at most $2$.
Thus $N$ is classical in its
natural representation or is alternating or symmetric on a fully
deleted permutation module. (Note that the $\mathrm{SU}_5(2)$ case listed in \cite{GuralnickSaxl} does not occur since we are working in characteristic $2$.)

Suppose first that the classical socle is linear, symplectic, or
unitary. In characteristic $2$, its natural representation contains a
transvection. Such a transvection belongs to $N$, contradicting Lemma~\ref{l:G0notransvection}. Hence only the orthogonal
case remains.


Let $f$ be a positive integer such that $N$ is defined as a natural orthogonal subgroup of $\mathrm{GL}_d(\mathbb{F}_{2^f})\leq \mathrm{GL}(\overline{V})$. 
Then \cite[Corollary~2.10.10]{KleidmanLiebeck} implies that $f=1$ thus we can assume that $N$ is embedded naturally in $\mathrm{GL}_d(2)$ as an orthogonal group. 
The only linear groups between
$\Omega_d^\pm(2)$ and its natural linear normalizer are
$\Omega_d^\pm(2)$ and $\mathrm O_d^\pm(2)$. The latter contains
orthogonal transvections, whence
$N=\Omega_d^\pm(2)$.

Suppose now that $N$ is of alternating or symmetric type. The fully
deleted module has dimension $c-1$ when $c$ is odd and $c-2$ when
$c$ is even. Hence $c=d+1$ or $c=d+2$ and in both cases $d$ is even.
A transposition acts as a transvection on the fully deleted
module, and so $N=\mathrm{Alt}(d+1)$ or $\mathrm{Alt}(d+2)$ by Lemma~\ref{l:G0notransvection}. Let $U$ be the fully deleted module of dimension $d$ over $\mathbb{F}_2$. By assumption  $\overline{V}$ and $k\otimes_{\mathbb{F}_2} U$ (the fully deleted module over $k$) are isomorphic as $kN$-modules. Therefore $V$ and $U$ are isomorphic as $\mathbb{F}_2N$-modules, as required.
\end{proof}

It remains to recover $G_0$ from the normal subgroup $N$.

\begin{lemma}\label{l:normalizers-complete}
One of the following holds:
\[
G_0\cong\Omega_d^+(2),\quad
G_0\cong\Omega_d^-(2),\quad
G_0\cong\mathrm{Alt}(d+1),\quad\mathrm{Alt}(d+2)
\]
in the corresponding natural or fully deleted representation. In particular, $d$ is even.
\end{lemma}

\begin{proof}
Since $N\unlhd G_0$, $G_0\leq{\bf N}_{\mathrm{GL}_d(2)}(N)$.

If $N=\mathrm{Alt}(c)$, where $c\in\{d+1,d+2\}$, then
${\bf N}_{\mathrm{GL}_d(2)}(N)=\mathrm{Sym}(c)$.
Indeed, the centralizer of $N$ is trivial by absolute irreducibility,
so conjugation embeds the normalizer into
$\mathrm{Aut}(\mathrm{Alt}(c))=\mathrm{Sym}(c)$; the standard
$\mathrm{Sym}(c)$ on the fully deleted module is already contained in
the normalizer. Thus either $G_0=N$ or $G_0=\mathrm{Sym}(c)$. The
second possibility contains a transposition, which acts as a
transvection, and is excluded by Lemma~\ref{l:G0notransvection}.

If $N=\Omega_d^\pm(2)$, then
${\bf N}_{\mathrm{GL}_d(2)}(N)=\mathrm O_d^\pm(2)$.
For completeness, let $Q$ be the invariant quadratic form. The polar
form of $Q$ is, up to scalar, the unique $N$-invariant nondegenerate
alternating form; over $\mathbb F_2$ the scalar is $1$. If an element
normalizes $N$, it sends $Q$ to another $N$-invariant quadratic form
with the same polar form. The difference of these two quadratic forms
has zero polar form and is the square of an $N$-invariant linear
functional. Since the natural $N$-module has no non-zero fixed
functional, the two quadratic forms are equal. Thus the normalizer is
contained in $\mathrm O_d^\pm(2)$, and the reverse containment is
standard. Since
$[\mathrm O_d^\pm(2):\Omega_d^\pm(2)]=2$, either $G_0=N$ or
$G_0=\mathrm O_d^\pm(2)$. The latter contains an orthogonal
transvection and is again excluded by Lemma~\ref{l:G0notransvection}.
\end{proof}

When $d\ge 6$, the examples $G_0=\Omega_d^\pm(2)$ do have relational complexity $d+1$ by Proposition~\ref{thrm:calculations}, giving part~\eqref{part32} of Theorem~\ref{thrm:main3}. 
The examples $G_0=\mathrm{Alt}(d+1)$ and $\mathrm{Alt}(d+2)$ also have relational complexity $d+1$ by Proposition~\ref{thrm:calculations}, giving part~\eqref{part22}  of Theorem~\ref{thrm:main3}. This completes the proof.
\section*{Declaration of AI use}
\noindent{\small{A large language model (Claude Opus 5.5) was used to audit the completed manuscript 
for any errors. It was not used to suggest repairs.}}
\end{document}